\documentclass{article}%
\usepackage{amsmath}
\usepackage{amsfonts}
\usepackage{amssymb}
\usepackage{graphicx}
\usepackage{scalefnt}%
\addtocounter{footnote}{1}
\makeatletter
\@addtoreset{figure}{section}
\def\thefigure{\thesection.\@arabic\c@figure}
\def\fps@figure{h,t}
\@addtoreset{table}{bsection}
\def\thetable{\thesection.\@arabic\c@table}
\def\fps@table{h, t}
\@addtoreset{equation}{section}

\makeatother
\newtheorem{theorem}{Theorem}

\newtheorem{corollary}[theorem]{Corollary}

\newtheorem{definition}[theorem]{Definition}

\newtheorem{lemma}[theorem]{Lemma}

\newtheorem{proposition}[theorem]{Proposition}

\newenvironment{proof}[1][Proof]{\noindent\textbf{#1.} }{\ \rule{0.5em}{0.5em}}
\newcommand{\bfi}{\bfseries\itshape}
\begin{document}

\title{\textbf{Knudsen's law and random billiards in irrational triangles}}
\author{Joan-Andreu L\'{a}zaro-Cam\'{\i}$^{1}$, Kamaludin Dingle$^{2}$, and Jeroen
Lamb$^{3}$.$\bigskip$\\$^{1,3}${\small Department of Mathematics. Imperial College London.}\\{\small 180 Queen's Gate, SW7 2AZ, London, UK.}\\$^{2}${\small Systems Biology DTC, Oxford University, OX1 3QD, UK.}}
\date{}
\maketitle

\begin{abstract}
We prove Knudsen's law for a gas of particles bouncing freely in a two
dimensional pipeline with serrated walls consisting of irrational triangles.
Dynamics are randomly perturbed and the corresponding random map studied under
a skew-type deterministic representation which is shown to be ergodic and exact.

\end{abstract}

\makeatletter\addtocounter{footnote}{1} \footnotetext{e-mail:
\texttt{j.lazaro-cami@imperial.ac.uk}} \addtocounter{footnote}{1}
\footnotetext{e-mail: \texttt{kamaludin.dingle@sjc.ox.ac.uk}}
\addtocounter{footnote}{1} \footnotetext{e-mail:
\texttt{jsw.lamb@imperial.ac.uk}} \makeatother

\bigskip


\bigskip

\section{Introduction\label{section introduction}}

In his nowadays classical studies on the kinetic theory of gases, the Danish
physicist M. Knudsen experimentally observed that, no matter how an inert gas
was injected into a pipeline, the direction in which a molecule rebounds from
the pipeline's solid wall is asymptotically independent of its initial
trajectory. That is, the fraction of particles leaving the surface in a given
direction is proportional to $\cos\left(  \theta\right)  $, where $\theta$ is
the angle that such a particle's trajectory defines, measured with respect to
the normal to the surface (\cite{knudsen book}). This behaviour is referred to
as the {\bfseries\itshape Knudsen's (cosine) law} ever since. In the
experiment, the gas is injected at a very low pressure so that interactions
between particles are negligible.

The physical justification of Knudsen's law is the following. First of all,
one assumes that particles bounce at the pipeline's wall \textit{elastically}.
This means that the energy of a particle is preserved in a collision which, in
turn, implies the \textit{refection law}: the angle that the incident
direction forms with the normal to a surface, $\theta_{in}$, equals the angle
$\theta_{ref}$ of the reflected direction, $\theta_{in}=\theta_{ref}$. It has
been proved that this law is valid in a first approximation if we do not take
into account thermal effects. Since we also assume that particles do not
interact with one another, we must conclude that the microscopic
irregularities on the pipeline's surface are responsible for the destruction
of any particular pattern in the original gas distribution. Indeed, even if we
assume that the bounces are perfectly elastic, microscopic holes in the
boundary of the pipeline and imperfections in relief are dimensionally
comparable to the molecules of the gas and have therefore a disruptive (i.e.,
unpredictable) effect on particle collisions: after many bounces, we are in
the so-called \textit{Knudsen's regime} in which the reflected direction is
independent of the incident one. However, this argument does not explain by
itself why the reflected angles are distributed according to the
\textit{cosine law}. The theory of \textit{billiards} helps clarify this point
(\cite{feres knudsen law}, \cite{feres random walks}).

The irregularities of the pipeline surface can be reasonably modelled as
cavities or microscopic cells with a dispersive geometry such that, once a
particle has entered one of them, it comes out of it with a rather arbitrary
direction, even if all the collisions inside that cavity are elastic. If the
pipeline is made of a uniform material, we can model one such cell as a
billiard table and the pipeline wall as an infinite row of such billiard
tables that a particle moving freely enters and exits. With this description,
Knudsen's cosine law can be then seen as a consequence of the fact that, for
sufficiently dispersive billiard geometries, the Liouville measure is the
unique measure preserved by the billiard flow (see \cite{chernov book} for
more details). Unfortunately, this need not be the case for polygonal (i.e.,
non-dispersive) tables such as Fig. \ref{fig pipeline}. So polygonal
geometries require a slightly different approach.

Since the characteristic size of this cell is infinitely small compared with
the diameter of the pipeline, the exact position along the open side of the
cell (dotted line in Fig. \ref{fig pipeline}) at which the particle enters a
cell coming from the previous one is considered to be randomly distributed
with uniform probability. Following \cite{feres knudsen law}, we will refer to
these billiards as \textit{random billiards}. More concretely, the dynamics of
a particle bouncing inside the pipeline are given by the \textit{first return
map} to the open side of a the billiard cell, which defines a Markov process.
Thus, the dynamics are characterised by a transition operator $K\left(
\theta,A\right)  $ such that, for any reflected direction $\theta\in
\lbrack0,\pi]$, $K\left(  \theta,A\right)  $ is the probability that a
molecule takes a direction in $A\subseteq\lbrack0,\pi]$ after the next
rebound. These concepts will be reviewed in Sections
\ref{section preliminaries} and \ref{section random billiards}. In this
formalism, Knudsen's law can be mathematically written as follows.

Let $\nu$ be the initial angle distribution with which the gas is injected
into the pipeline. That is, if $A\subseteq\lbrack0,\pi]$, $\nu(A)$ represents
the proportion of particles that will hit the pipeline's wall for the first
time with an incident angle $\theta\in A$. Denote by $\nu^{(n)}$ the
distribution after $n$ collisions. In this context, Knudsen's cosine law can
be expressed in two different ways. The {\bfseries\itshape strong Knudsen's
law} claims that, for almost any initial distribution $\nu$, after many
collisions $\nu^{(n)}$ converges to the cosine distribution,
\begin{equation}
\nu^{(n)}(A)\underset{n\rightarrow\infty}{\longrightarrow}\frac{1}{2}\int
_{A}\cos\left(  \theta\right)  d\theta,\text{ \ }A\subseteq\lbrack0,\pi].
\label{eq strong law}%
\end{equation}
On the other hand, the {\bfseries\itshape weak Knudsen's law} states that if,
after any collision, we count the number of particles reflected with a given
angle $\theta$ and then divide by the total number of particles, this quantity
is proportional to $\cos(\theta)$. Explicitly,%
\begin{equation}
\frac{1}{n}\sum_{i=1}^{n}\nu^{(i)}(A)\underset{n\rightarrow\infty
}{\longrightarrow}\frac{1}{2}\int_{A}\cos\left(  \theta\right)  d\theta.
\label{eq weak law}%
\end{equation}
Obviously (\ref{eq strong law}) implies (\ref{eq weak law}), so the strong
Knudsen's law implies the weak one.

Deterministic dynamical systems can be sometimes understood as idealizations
of \textit{real} systems, the latter usually subjected to negligible
perturbations in a first approximation. The use of a dynamical approach to
study a gas is one such example. In this paper, we will study the properties
of the random billiard obtained as a result of modelling the rough surface of
the pipeline by a zig-zag geometry as shown in Fig. \ref{fig pipeline}. Our
billiard cell is an isosceles triangle with open side $\overline{pq}$ defined
by an angle $\alpha$. This model was introduced in \cite{feres random walks}
but has not been studied in depth yet. For example, in \cite{feres random
walks} nothing is said about the ergodic properties of this billiard table
when $\alpha/2\pi\notin\mathbb{Q}$ (it is shown that it is not ergodic under
the Liouville measure if $\alpha/2\pi\in\mathbb{Q}$ though). Thus it is not at
all clear if the dynamics resulting from a random billiard as in Fig.
\ref{fig billiard table} will in fact follow Knudsen's law (weak or strong).
In this paper, we will prove that it is actually the case. More concretely,
the main contributions of our paper are the following:

\begin{enumerate}
\item Little is known about the ergodicity of the flow of deterministic
billiards given by irrational polygon tables (\cite{gutkin open problems}).
Remarkably, we show that, for a concrete example, randomising the billiard in
a sensible way (i.e., choosing the point at which the particle enters the
billiard randomly but uniformly distributed on one side of the table) implies
its ergodicity and exactness with respect to the Liouville measure.

\item Unlike other attempts to prove (weak) Knudsen's law, where one usually
assumes very irregular geometries at a microscopic level responsible for the
dispersive effects, we deal with an extremely simple pattern. This has
additional advantages as a simpler model capturing the main features of a
system can be simulated and developed more easily.

\item In the literature, ergodicity of random billiards relies on the
ergodicity of the first return map (see \cite{feres random walks}). We use
here a different approach and take advantage of a skew-type representation of
random maps recently introduced in \cite{bashoun}. The techniques used to
prove the exactness of our model can be certainly adapted to study other
(random) billiard tables and random maps in general.

\item Finally, we prove that the strong Knudsen's law holds for our model. As
we mentioned before, even if one assumes dispersive billiard geometries, one
will only prove that the first return map is ergodic, i.e., the weak law.
\end{enumerate}

The paper is structured as follows: in Section \ref{section model} we
introduce the random billiard behind our model. We recall some basic
definitions and properties of billiards in Section \ref{section preliminaries}
and the concept of random billiard in Section \ref{section random billiards}.
In Section \ref{section random maps}, we review the skew-type representation
of random maps introduced in \cite{bashoun} and show the relationship between
this representation and the dynamical evolution of absolutely continuous
measures. We prove that this skew-type representation is exact in Section
\ref{section exactness of S}, which uses some auxiliary results presented
separately in Section \ref{section ergodicity of SS}. Finally, in Section
\ref{section knudsen's law}, we illustrate our results with numerical simulations.

\section{The Model\label{section model}}

Suppose we have a gas of non-interacting particles moving freely inside a
pipeline. Our model consists of a two dimensional infinite pipeline whose
rough walls are modelled as a sequence of cells built as a juxtaposition of a
\textit{fundamental cell }(Fig. \ref{fig billiard table}). In this section, we
are going to fix the geometry of such cell and introduce the main hypothesis
behind the dynamics of the particles. The content of this section is extracted
from \cite{feres random walks}.

Both the geometry and dynamics of our model are summarised in Figure
\ref{fig pipeline}. We assume that the walls of the pipeline are not smooth
but describe a serrated regular pattern built from a fixed isosceles triangle
with an open side. We characterise that triangle by one of its angles $\alpha
$. The dynamics of a particle moving freely inside the pipeline are as
follows: the particle enters one of the cells with some angle $\theta_{in}$
and bounces on its sides elastically, that is, according to the reflection
law. The particle eventually comes out of the cell with a given direction
$\theta_{out}$, then crosses the pipeline, and reaches another cell on the
opposite wall, and the process recurs. Since the diameter of the pipeline is
several orders of magnitude bigger than the characteristic length
$\overline{pq}$ of the cell, it is plausible to think that, every time the
particle enters a new cell through the open side, it does so at a point
uniformly distributed on $\overline{pq}$. \vspace{-0.5cm}

\begin{figure}[h]
\begin{center}
\includegraphics[height=6cm,width=12.5cm]{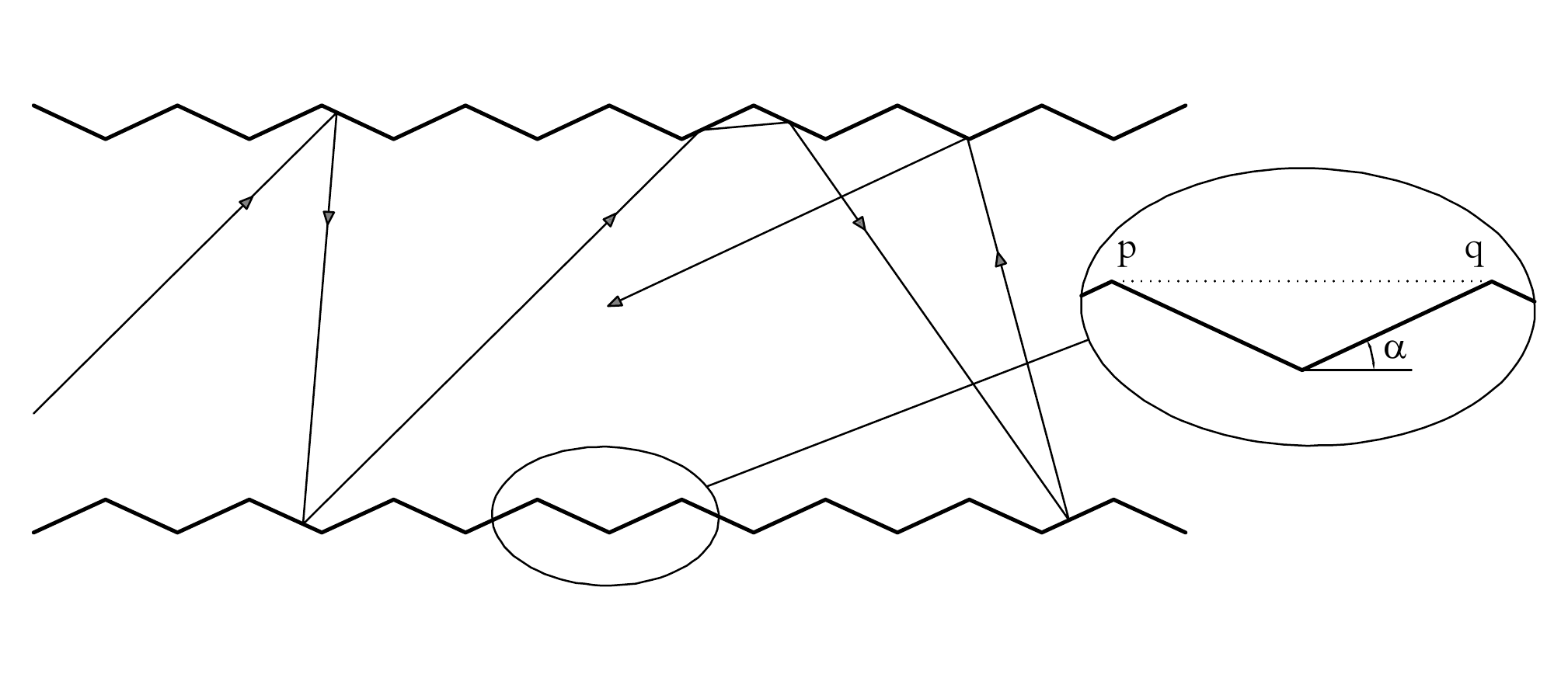}
\end{center}
\par
\vspace{-1cm} \caption{Particle bouncing in a pipeline with serrated
triangular boundaries.}%
\label{fig pipeline}%
\end{figure}

In the literature, one encounters examples of particle dynamics on bounded
domains where the random perturbation is introduced in the observations of
$\theta_{out}$ every time the particles bounces off the wall (see for example
\cite{comets}), which implies that collisions are no longer elastic. This
perturbation is explained as a consequence of the roughness of the wall.
Observe that, in our model, randomness is introduced in a completely different
way. We assume elastic collisions but, instead of considering dispersive
geometries responsible for the random bounces, the noise is introduced in a
sensible way without modifying the (rather elementary) geometry of the problem.

The trajectory that a particle describes after it comes out of a billiard cell
until it enters another one on the opposite wall is completely irrelevant for
dynamical purposes. In practice, we can better understand our dynamical system
as the closed billiard table obtained by fixing together two triangular cells
so that they share the same open side as in Figure \ref{fig billiard table}.
When the particle crosses that open side, it retains its direction but it is
assigned a different position on $\overline{pq}$ according to a uniform
law.\vspace{-0.5cm}

\begin{figure}[h]
\begin{center}
\includegraphics[height=4.8cm,width=10cm]{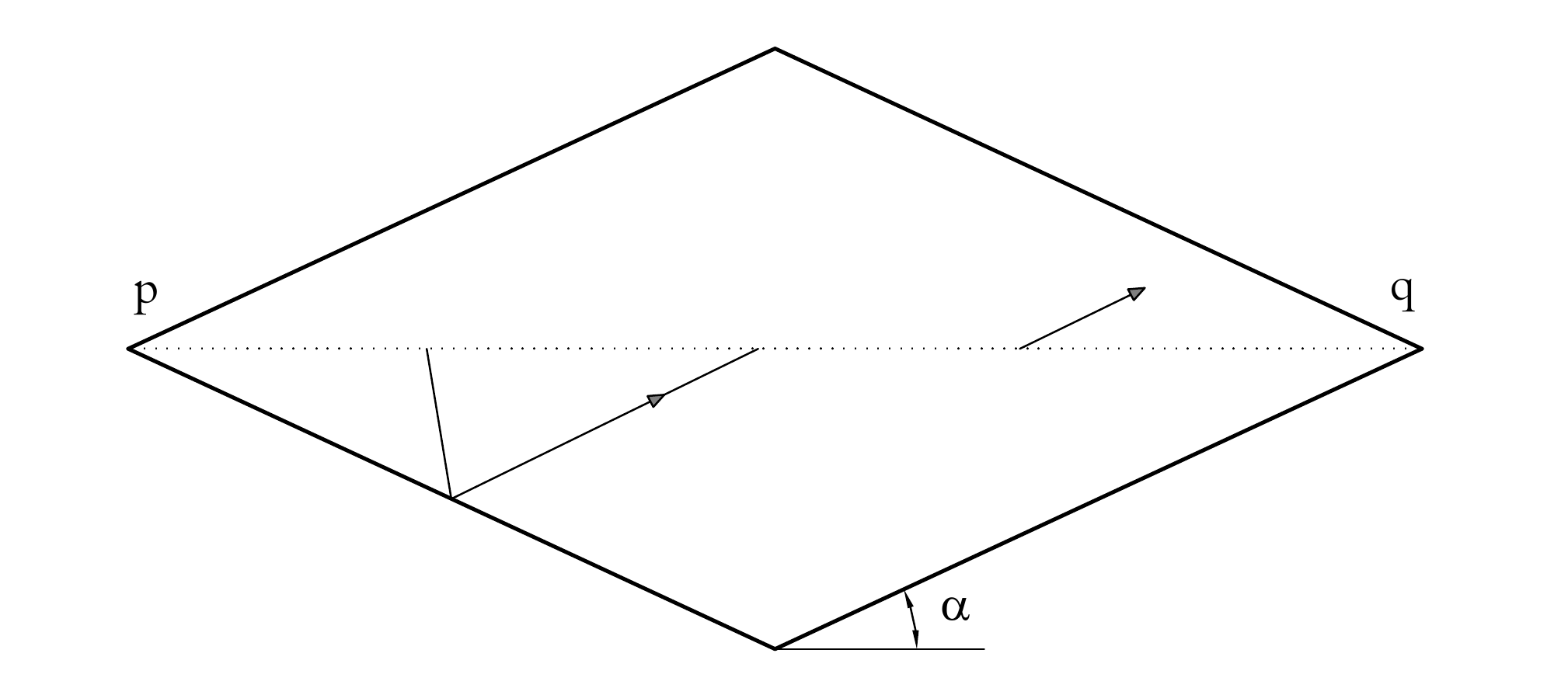}
\end{center}
\par
\vspace{-0.5cm} \caption{Fundamental cell. A particle crosses the open side
$\overline{pq}$ at a uniformly distributed random point.}%
\label{fig billiard table}%
\end{figure}

The {\bfseries\itshape incoming angle} $\theta_{in}$ (resp.
{\bfseries\itshape
outgoing angle}) is then the angle in $[-\pi/2,\pi/2]$ with which the particle
enters (resp. leaves) the cell measured with respect to the normal to the open
side $\overline{pq}$. One often shifts $[-\pi/2,\pi/2]$ by $\pi/2$ and measure
$\theta$ in $[0,\pi]$, i.e. with respect to horizontal. In \cite{feres random
walks}, for any incoming angle $\theta\in\lbrack0,\pi]$, Feres gives the
different possible directions with which a particle may come out of the
fundamental cell (Fig. \ref{fig billiard table}). They are four, given by maps
$\tau_{i}:[0,\pi]\rightarrow\mathbb{R}$ defined as%
\begin{align}
\tau_{1}\left(  \theta\right)   &  =\theta+2\alpha,~~\tau_{2}\left(
\theta\right)  =-\theta+2\pi-4\alpha\nonumber\\
\tau_{3}\left(  \theta\right)   &  =\theta-2\alpha,~~\tau_{4}\left(
\theta\right)  =-\theta+4\alpha, \tag{M1}\label{eq tau maps}%
\end{align}
and any $\tau_{i}$ applies with certain probability $p_{i}$. The probabilities
associated to the maps (\ref{eq tau maps}) are as follows:%
\begin{align}
p_{1}\left(  \theta\right)   &  =\left\{
\begin{array}
[c]{ll}%
1 & \theta\in\lbrack0,\alpha)\\
u_{\alpha}\left(  \theta\right)  & \theta\in\lbrack\alpha,\pi-3\alpha)\\
2\cos(2\alpha)u_{2\alpha}(\theta) & \theta\in\lbrack\pi-3\alpha,\pi-2\alpha)\\
0 & \theta\in\lbrack\pi-2\alpha,\pi]
\end{array}
\right.  \smallskip\nonumber\\
p_{2}\left(  \theta\right)   &  =\left\{
\begin{array}
[c]{ll}%
0 & \theta\in\lbrack0,\pi-3\alpha)\\
u_{\alpha}\left(  \theta\right)  -2\cos(2\alpha)u_{2\alpha}(\theta) &
\theta\in\lbrack\pi-3\alpha,\pi-2\alpha)\\
u_{\alpha}\left(  \theta\right)  & \theta\in\lbrack\pi-2\alpha,\pi-\alpha)\\
0 & \theta\in\lbrack\pi-\alpha,\pi]
\end{array}
\right.  \smallskip\nonumber\\
p_{3}\left(  \theta\right)   &  =\left\{
\begin{array}
[c]{ll}%
0 & \theta\in\lbrack0,2\alpha)\\
2\cos(2\alpha)u_{2\alpha}(-\theta) & \theta\in\lbrack2\alpha,3\alpha)\\
u_{\alpha}\left(  -\theta\right)  & \theta\in\lbrack3\alpha,\pi-\alpha)\\
1 & \theta\in\lbrack\pi-\alpha,\pi]
\end{array}
\right.  \smallskip\nonumber\\
p_{4}\left(  \theta\right)   &  =\left\{
\begin{array}
[c]{ll}%
0 & \theta\in\lbrack0,\alpha)\\
u_{\alpha}\left(  -\theta\right)  & \theta\in\lbrack\alpha,2\alpha)\\
u_{\alpha}\left(  -\theta\right)  -2\cos(2\alpha)u_{2\alpha}(-\theta) &
\theta\in\lbrack2\alpha,3\alpha)\\
0 & \theta\in\lbrack3\alpha,\pi],
\end{array}
\right.  \tag{M2}\label{eq probabilities model}%
\end{align}
where%
\begin{equation}
u_{\alpha}\left(  \theta\right)  =\frac{1}{2}\left(  1+\frac{\tan\alpha}%
{\tan\theta}\right)  \label{eq definition function u}%
\end{equation}
and $\alpha$ is assumed to be smaller than $\pi/6$. As we will see later, this
family of maps and probabilities is all we need to prove Knudsen's law.

\section{Preliminaries on Billiard Dynamics\label{section preliminaries}}

This section aims at recalling the main concepts of the theory of
(deterministic) billiards and one of its key results, namely, that the
\textit{billiard map} preserves the Liouville measure. The content of this
section was extracted from \cite{chernov book}, which the reader is referred
to for an exhaustive exposition on billiards.

A {\bfseries\itshape billiard table} $D$ is the closure of a bounded open
domain $D_{0}\subset\mathbb{R}^{2}$ whose boundary $\partial D=\Gamma_{1}%
\cup\cdots\cup\Gamma_{m}$ is a finite union of smooth compact curves ($C^{l}$,
$l\geq3$). $\Gamma_{1},...,\Gamma_{n}$ are called the {\bfseries\itshape
walls} of the billiard table and are assumed to intersect each other only at
their endpoints or {\bfseries\itshape corners}. For any $i=1,...,n$,
$\Gamma_{i}$ is defined by a $C^{l}$ map $f_{i}:[a_{i},b_{i}]\subset
\mathbb{R}\rightarrow\mathbb{R}^{2}$ such that the second derivative
$f_{i}^{\prime\prime}$ either vanishes or is identically zero. This condition
prevents the wall $\Gamma_{i}$ from having inflection points or line segments.
A wall such that $f^{\prime\prime}\neq0$ is called {\bfseries\itshape
focusing} or {\bfseries\itshape dispersing} if $f^{\prime\prime}$ points
inwards or outwards $D$, respectively. Otherwise $f^{\prime\prime}=0$ and the
wall is called {\bfseries\itshape flat}. In this framework, we consider the
dynamics of a point particle moving freely within $D$.

Let $q\in D$ denote the position of a free particle moving within a billiard
table. We suppose that the particle bounces at $\partial D$
{\bfseries\itshape
elastically}. That is, if it collides with the wall $\Gamma$ at a point $p$
that is not a corner, the incident angle $\theta_{in}$ that the velocity $v$
forms with the normal vector $n$ to $\Gamma$ at $p$ equals the angle of
reflection $\theta_{ref}$ between $n$ and $v$ after the collision, i.e.
$\theta_{in}=\theta_{ref}$. This implies that $v$ has constant norm, which we
will assume equal to $1$ for the sake of simplicity. Under these assumptions,
the {\bfseries\itshape billiard flow} $\Phi$ is the flow that this particle
defines on the {\bfseries\itshape phase space} $D\times\mathbb{S}^{1}$. Since
the particle follows a straight trajectory between two consecutive collisions,
the dynamical information of the system is contained in the geometry of the
boundary $\partial D$ and how the particle bounces off it. Therefore, instead
of studying the billiard flow, first one introduces the cross-section
$M=\partial D\times\lbrack-\pi/2,\pi/2]$ as the set of all postcollisional
velocity vectors, where $\theta\in\lbrack-\pi/2,\pi/2]$ measures the angle
between $v$ and the normal vector $n$ pointing to the interior of $D$, and
then considers the first return map $F$ that the billiard flow induces on $M$.
That is, given a point of $M$, $F$ gives the next position at where the
particle collides and its velocity after that collision. The map $F$ is often
called the {\bfseries\itshape billiard map}.

It is a well known result in the theory of billiards that the billiard map
preserves the {\bfseries\itshape Liouville measure} $\lambda\otimes\mu$ on
$M$, where $\lambda$ is the Lebesgue measure on $\partial D$ and $\mu$ is the
measure on $[-\pi/2,\pi/2]$ given by $\mu(A)=%
\frac12
\int_{A}\cos(\theta)d\theta$, $A\subseteq\mathcal{B}\left(  [-\pi
/2,\pi/2]\right)  $ (\cite[Lemma 2.35]{chernov book}). However, except for a
few general results, little is known about the ergodicity of billiard maps
with respect to the Liouville measure for explicit examples of boundary
geometries. For example, dispersive billiards are ergodic (\cite[Theorem
6.20]{chernov book}) while regular polygons are not.

Finally, as we pointed out in Section \ref{section model}, it is customary in
the literature to translate $[-\pi/2,\pi/2]$ and measure $\theta$ in $[0,\pi
]$, in which case $\mu$ is given by $\mu(A)=%
\frac12
\int_{A}\sin(\theta)d\theta$, $A\subseteq\mathcal{B}\left(  [0,\pi]\right)  $.
We will follow this convention throughout the paper.

\section{First Return Map. Random Billiards\label{section random billiards}}

In this section we are going to get back to the concept of \textit{random
billiard} and its different representations.

Let $D$ be a billiard table with boundary $\partial D=\Gamma_{1}\cup\cdots
\cup\Gamma_{m}$. Remove one of the walls of $\partial D$, for example
$\Gamma_{1}$, so that we obtain a billiard table with an \textit{open side}.
Without loss of generality, we can assume that the open side is just a line
segment $\overline{pq}$ as in Figure \ref{fig billiard table}. This
identification can be carried out by the map $f_{1}:[a_{1},b_{1}%
]\subset\mathbb{R}\rightarrow\mathbb{R}^{2}$ defining $\Gamma_{1}$. On the
other hand, suppose a particle enters the billiard from the open side at
$x\in\overline{pq}$ with an angle $\theta_{in}\in\lbrack0,\pi]$ with respect
to $\Gamma_{1}$. Denote by $\Psi_{x}(\theta_{in})\in\lbrack0,\pi]$ the angle
associated to the velocity of that particle when it returns to the open side
for the first time. The {\bfseries\itshape first return map} (to the open
side) is then the map $T:\overline{pq}\times\lbrack0,\pi]\rightarrow
\overline{pq}\times\lbrack0,\pi]$ induced from the billiard map $F$ such that
$T\left(  x,\theta_{in}\right)  =\left(  y,\Psi_{x}(\theta_{in})\right)  $,
i.e., the particle leaves the billiard at $y\in\overline{pq}$ with an angle
$\theta_{out}=\Psi_{x}(\theta_{in})$.

The Liouville measure $\lambda\otimes\mu$ on $\partial D\times\lbrack0,\pi]$
induces a measure on $\overline{pq}\times\lbrack0,\pi]$ by restriction in a
natural way. We will continue referring to this measure as the Liouville
measure and denoting it by $\lambda\otimes\mu$, where $\lambda$ now stands for
the Lebesgue measure on $\overline{pq}$. It can be proved that the first
return map $T$ also preserves the Liouville measure (\cite[Exercise
I.5.1]{chernov billiard notes}). Furthermore,

\begin{lemma}
\label{lemma ergodicity first return}If the billiard map $F:\partial
D\times\lbrack0,\pi]\rightarrow\partial D\times\lbrack0,\pi]$ is ergodic with
respect to $\lambda\otimes\mu$, so is the first return map $T:\overline
{pq}\times\lbrack0,\pi]\rightarrow\overline{pq}\times\lbrack0,\pi]$.
\end{lemma}

\begin{proof}
Let $A\in\mathcal{B}(\overline{pq})\otimes\mathcal{B}([0,\pi])$ be an
invariant set, i.e., $T^{-1}\left(  A\right)  =A$. This set can be naturally
regarded as a measurable set of $\partial D\times\lbrack0,\pi]$ since
$\overline{pq}\subset\partial D$ is identified with the wall $\Gamma_{1}$.
Since $F$ is ergodic, $%
{\textstyle\bigcup\nolimits_{n\geq1}}
F^{-n}\left(  A\right)  $ has full $\left.  \lambda\otimes\mu\right\vert
_{\partial D\times\lbrack0,\pi]}$-measure, which means that%
\[
\left.  \lambda\otimes\mu\right\vert _{\partial D\times\lbrack0,\pi]}((%
{\textstyle\bigcup\nolimits_{n\geq1}}
F^{-n}\left(  A\right)  )\bigcap\left(  \overline{pq}\times\lbrack
0,\pi]\right)  )=1.
\]
If $y\in F^{-n}\left(  A\right)  \bigcap\left(  \overline{pq}\times
\lbrack0,\pi]\right)  $, then $y=F^{-n}\left(  x\right)  $ for some $x\in A$
and, from the very definition of the first return map that, $y=T^{-m_{n}%
}\left(  x\right)  $ for some $m_{n}\leq n$. Therefore,%
\[
\left(  \bigcup\nolimits_{n\geq1}F^{-n}\left(  A\right)  \right)
\bigcap\left(  \overline{pq}\times\lbrack0,\pi]\right)  =\bigcup
\nolimits_{m\geq1}T^{-m}\left(  A\right)  =A,
\]
where the last equality follows from the invariance of $A$. Consequently,
$\left.  \lambda\otimes\mu\right\vert _{\overline{pq}\times\lbrack0,\pi
]}(A)=1$ and $T$ is ergodic.\smallskip
\end{proof}

Let $\left(  \overline{pq},\mathcal{B}(\overline{pq}),\lambda\right)  $ be the
probability space built upon $\overline{pq}$ with the normalised Lebesgue
measure $\lambda$ on the Borel $\sigma$-algebra $\mathcal{B}(\overline{pq})$.
The {\bfseries\itshape random billiard} associated to $D$ is then the
time-discrete random dynamical system $\left\{  \Psi_{x}:[0,\pi]\rightarrow
\lbrack0,\pi]~|~x\in\overline{pq}\right\}  $, where $\overline{pq}$ is
regarded as a probability space. Roughly speaking, the random billiard thus
built models the outgoing state a free particle that entered the billiard
through the open side at a point uniformly distributed along $\overline{pq}$.
When regarded as a probability space, $\overline{pq}$ will be denoted by
$\Omega$.

The random billiard $\left\{  \Psi_{x}:[0,\pi]\rightarrow\lbrack0,\pi
]~|~x\in\Omega\right\}  $ defines a transition probability kernel
$K:[0,\pi]\times\mathcal{B}\left(  [0,\pi]\right)  \rightarrow\lbrack0,1]$
given by%
\begin{equation}
K(\theta,A)=\lambda\left(  \{x\in\Omega~%
\vert
~\Psi_{x}(\theta)\in A\}\right)  \text{, where }x\in\lbrack0,\pi]\text{ and
}A\in\mathcal{B}\left(  [0,\pi]\right)  . \label{eq probability kernel}%
\end{equation}
We say that $\nu$ is {\bfseries\itshape invariant} with respect to $\left\{
\Psi_{x}:[0,\pi]\rightarrow\lbrack0,\pi]~|~x\in\Omega\right\}  $ if
\[
\nu(A)=\int_{[0,\pi]}K(\theta,A)d\nu(\theta)\text{ for any }A\in
\mathcal{B}\left(  [0,\pi]\right)
\]
and that the random billiard is {\bfseries\itshape ergodic} with respect to
$\nu$ if $T$ is ergodic with respect to $\lambda\otimes\nu$. For example,
$\mu(A)=%
\frac12
\int_{A}\sin\theta d\theta$, $A\in\mathcal{B}\left(  [0,\pi]\right)  $, is an
invariant measure because $\lambda\otimes\mu$ is invariant with respect to the
first return map $T$ (\cite[Proposition 2.1]{feres random walks}). Moreover,
since dispersive billiards are ergodic with respect to the Liouville measure,
the corresponding random billiard is also ergodic with respect to $\mu$ (Lemma
\ref{lemma ergodicity first return}). In this situation, Birkhoff's Ergodic
Theorem implies that
\begin{equation}
\frac{1}{n}\sum_{i=1}^{n}\mathbf{1}_{A}\left(  \pi_{2}\circ T^{n}\right)
\underset{n\rightarrow\infty}{\longrightarrow}%
\frac12
\int_{A}\sin(\theta)d\theta\text{~~a.s.,} \label{eq 18}%
\end{equation}
where $\pi_{2}:\overline{pq}\times\lbrack0,\pi]\rightarrow\lbrack0,\pi]$ is
the projection onto the second factor. That is, the average number of
particles reflected with some angle $\theta\in A$ converges to $%
\frac12
\int_{A}\sin(\theta)d\theta$ so (weak) Knudsen's law holds for dispersive billiards.

In general, for non-dispersive billiards, little is know about the ergodicity
of the first return map. This is indeed the case of the random billiard
introduced in Section \ref{section model} (Fig. \ref{fig billiard table}).
Recall that, as far as we know, determining whether a general irrational
triangular billiard is ergodic is still an open problem (\cite[Section
7]{gorodnik open problmes}) and only a few particular examples were proved to
be ergodic (\cite{vorobets}). Therefore, at this point, it is not clear at all
if Knudsen's law holds for the pipeline model introduced in Section
\ref{section model}. To prove it, the first return map must be replaced with a
more convenient representation.

\section{Skew-Type Representation of Random Maps\label{section random maps}}

Given a random billiard, the probability kernel (\ref{eq probability kernel})
contains all the dynamical information of the system and determines its
properties. This implies that, as far as the dynamics is concerned, the
underlying probability space $\Omega$ (which in Section
\ref{section random billiards} was equal to $\left(  \overline{pq}%
,\mathcal{B}(\overline{pq}),\lambda\right)  $) plays a secondary role. In
particular, other probability spaces $\left(  \Omega,\mathcal{F},P\right)  $
giving rise to the same probability kernel are available. Indeed, if
$K:X\times\mathcal{B}\rightarrow\lbrack0,1]$ is a transition probability
kernel on a measurable space $\left(  X,\mathcal{B}\right)  $, Kolmogorov's
Existence Theorem on Markov processes guarantees that there exists a
probability $P$ defined on the Borel $\sigma$-algebra of the topological space
$\Omega=%
{\textstyle\prod\nolimits_{i=1}^{\infty}}
X_{i}$, $X_{i}:=X$ for all $i\geq1$, such that the chain $\Phi_{i}:=\pi_{i}$
is Markovian and has transition probability kernel $K$ (\cite[Theorem
2.11]{blumenthal book}), where $\pi_{i}:%
{\textstyle\prod\nolimits_{i=1}^{\infty}}
X_{i}\rightarrow X$ is the projection onto the $i$th factor. In this
statement, $X$ has to be a $\sigma$-compact Hausdorff space and $\mathcal{B}$
the Borel $\sigma$-algebra. Unfortunately, $%
{\textstyle\prod\nolimits_{i=1}^{\infty}}
X_{i}$ is not very manageable to work with explicitly. For instance, the
representation $T:\overline{pq}\times\lbrack0,\pi]\rightarrow\overline
{pq}\times\lbrack0,\pi]$ given by the first return map introduced in Section
\ref{section random billiards} only involves finite dimensional spaces and
hence seems more convenient. We are now going to consider another skew-type
representation for random maps recently introduced in \cite{bashoun}. As we
will show in Section \ref{section exactness of S}, this representation will be
crucial to prove the asymptotic properties of the random map defined by our
random billiard Fig. \ref{fig billiard table}.

Let $\left(  X,\mathcal{B},\nu\right)  $ be a general measure space and let
$\left(  [0,1),\mathcal{B}\left(  [0,1)\right)  ,\lambda\right)  $ be the unit
interval regarded as a probability space, where $\lambda$ stands for the
Lebesgue measure. Later on, we will apply the results of this section to
$X=[0,\pi]$. For any $k=1,...,N$, let $\tau_{k}:X\rightarrow X$ and
$p_{k}:X\rightarrow\lbrack0,1]$ be measurable mappings such that
$\{p_{k}\}_{k=1,...,N}$ is a measurable partition of the unity, i.e.,
$\sum_{k=1}^{N}p_{k}(x)=1$ for any $x\in X$ (compare to
(\ref{eq probabilities model}) and (\ref{eq tau maps}) in Section
\ref{section model}). We define the {\bfseries\itshape random dynamical
system} $\tau:X\rightarrow X$ such that $\tau\left(  x\right)  =\tau
_{i}\left(  x\right)  $ with probability $p_{i}\left(  x\right)  $. The
transition probability kernel of $\tau$ is given by%
\begin{equation}
K\left(  x,A\right)  =\sum_{i=1}^{N}p_{i}\left(  x\right)  \mathbf{1}%
_{A}\left(  \tau_{i}(x)\right)  .\label{eq 16}%
\end{equation}
This probability kernel defines the evolution of an initial distribution
(probability measure) $\nu$ on $\left(  X,\mathcal{B}\right)  $ under the
random map $\tau$ iteratively as%
\begin{equation}
\nu^{(0)}:=\nu,~~\nu^{(n+1)}(A)=\int_{[0,\pi]}K(\theta,A)d\nu^{(n)}%
(\theta),~~n\geq1,\label{eq 15}%
\end{equation}
where $A\in\mathcal{B}$. A measure $\nu$ is called {\bfi invariant} if
$\nu^{(1)}=\nu$.

Following \cite{bashoun}, consider now $\Omega=[0,1)\times X$ and set
$J_{k}:=\{(y,x)\in\Omega~%
\vert
~\sum_{i<k}p_{i}(x)\leq y<\sum_{i\leq k}p_{i}(x)\}$ such that $\{J_{k}%
\}_{k=1,...,N}$ is a finite measurable partition of $\Omega$. We define the
{\bfseries\itshape skew-type representation} of the random map $\tau$ as the
map $S:[0,1)\times X\rightarrow\lbrack0,1)\times X$ such that%
\begin{equation}
S\left(  y,x\right)  =\left(  \varphi_{k}(y,x),\tau_{k}(x)\right)  \text{ for
}\left(  y,x\right)  \in J_{k} \label{eq 12}%
\end{equation}
where%
\[
\varphi_{k}\left(  y,x\right)  =\frac{1}{p_{k}(x)}\left(  y-\sum_{i=1}%
^{k-1}p_{i}(x)\right)  .
\]
The map $\varphi_{k}$ is well defined because if $\left(  y,x\right)  \in
J_{k}$ then $p_{k}(x)>0$ necessarily. Moreover, $S$ thus defined is
$\mathcal{B}\left(  [0,1)\right)  \otimes\mathcal{B}$-measurable (\cite[Lemma
3.1]{bashoun}) and it is a skew-type representation of the random dynamical
system $\tau$.

As we will see, the map $S$ is extremely useful to study the properties of
$\tau$. For example, $\nu$ is an invariant measure of $\tau$ on $\left(
X,\mathcal{B}\right)  $ if and only if $\lambda\otimes\nu$ is an invariant
measure of $S$ on $\left(  [0,1)\times X,\mathcal{B}\left(  [0,1)\right)
\otimes\mathcal{B},\lambda\otimes\nu\right)  $ (\cite[Lemma 3.2]{bashoun}).
Moreover, since we know that the sine $\mu(A)=%
\frac12
\int_{A}\sin(\theta)d\theta$ is invariant by the transition probability kernel
(\ref{eq 16}) of our billiard table, we conclude that $\mu$ is invariant by
$\tau$ built from (\ref{eq probabilities model}) and (\ref{eq tau maps}), and
therefore $\lambda\otimes\mu$ is invariant by the corresponding map $S$. In an
abuse of terminology, we will continue calling $\lambda\otimes\mu$ the
{\bfseries\itshape Liouville measure when referring to }$S$. In this context,
we will say that the random map $\tau$ with initial distribution $\nu$ is
{\bfseries\itshape ergodic} (resp. {\bfseries\itshape mixing},
{\bfseries\itshape exact}) if $S$ is ergodic (resp. mixing, exact) with
respect to $\lambda\otimes\nu$. In Section \ref{section exactness of S}, we
are going to prove that the $S$ associated to (\ref{eq probabilities model})
and (\ref{eq tau maps}) is exact.%

\begin{figure}
[h]
\begin{center}
\includegraphics[
natheight=2.153400in,
natwidth=2.674000in,
height=2.1923in,
width=2.7155in
]%
{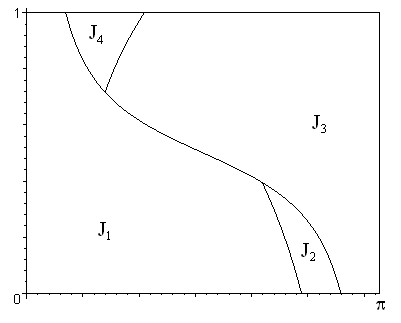}%
\caption{The sets $J_{i}$, $i=1,...,4$, for the random dynamical system
defined in Section \ref{section model}.}%
\label{fig S slices}%
\end{center}
\end{figure}

Unlike the first return map $T$ (Section \ref{section random billiards}), the
map $S$ has no dynamical interpretation and is a purely auxiliary tool to
represent and study $\tau$. Therefore, we need to establish the relationship
between $S$ and the transition probability kernel (\ref{eq 16}), which carries
all the dynamical information. This is the content of Theorem \ref{theorem 1}.
Before stating that relationship, we need to introduce some notation and an
auxiliary lemma, whose proof is included separately in the Appendix for the
sake of a clearer exposition.

\begin{definition}
\label{definition slices}We define iteratively the sets $J_{i_{1}\cdots i_{n}%
}$, where $i_{j}\in\{1,...,N\}$ for any $j=1,...,n$, as%
\begin{equation}
J_{i_{1}\cdots i_{n}}:=\left\{  \omega\in J_{i_{n}}~|~S\left(  \omega\right)
\in J_{i_{n-1}},S^{2}\left(  \omega\right)  \in J_{i_{n-2}},....,S^{n-1}%
(\omega)\in J_{i_{1}}\right\}  . \label{eq 27}%
\end{equation}

\end{definition}

\begin{lemma}
\label{lemma slices go to zero}Let $I_{x}:=[0,1)\times\{x\}$ the fibre through
$x\in X$. \ If $I_{x}\bigcap J_{i_{1}\cdots i_{n}}\neq\emptyset$, then%
\[
\lambda\left(  \pi_{1}\left(  I_{x}\bigcap J_{i_{1}\cdots i_{n}}\right)
\right)  =p_{i_{n}}(x)p_{i_{n-1}}\left(  \tau_{i_{n}}(x)\right)  \cdots
p_{i_{1}}(\tau_{i_{2}}\circ\cdots\circ\tau_{i_{n}}(x)).
\]

\end{lemma}

\begin{theorem}
\label{theorem 1}Let $\nu$ be a measure on $\left(  X,\mathcal{B}\right)  $
and $A\in\mathcal{B}$. Then,%
\[
\operatorname*{E}\nolimits_{\lambda\otimes\nu}\left[  \mathbf{1}_{A}\left(
\pi_{2}\circ S^{n}\right)  \right]  :=\int_{[0,1)\times X}\mathbf{1}%
_{A}\left(  \pi_{2}\circ S^{n}\right)  d\left(  \lambda\otimes\nu\right)
=\nu^{(n)}\left(  A\right)  .
\]

\end{theorem}

\begin{proof}
First of all, it is not difficult to check that%
\[
\nu^{(n)}\left(  A\right)  =\sum_{i_{1},...,i_{n}}\int dv\left(  x\right)
p_{i_{n}}\left(  x\right)  p_{i_{n-1}}\left(  \tau_{i_{n}}\left(  x\right)
\right)  \cdots p_{i_{1}}\left(  \tau_{i_{2}}\circ...\circ\tau_{i_{n}%
}(x)\right)  \mathbf{1}_{A}\left(  \tau_{i_{1}}\circ...\circ\tau_{i_{n}%
}(x)\right)  .
\]
This can be proved iteratively rewriting $K$ in terms of (auxiliary) Dirac
deltas, $K\left(  x,dy\right)  =\sum_{i=1}^{K}p_{i}\left(  x\right)
\delta\left(  y-\tau_{k}\left(  x\right)  \right)  dy$ where $\int f\left(
y\right)  \delta\left(  y-a\right)  dy=f\left(  a\right)  $, $f\in C\left(
\mathbb{R}\right)  $.

On the other hand, since $[0,1)\times X=\bigcup\nolimits_{i_{1},...,i_{n}%
}J_{i_{1}...i_{n}}$ is the disjoint union of the sets $J_{i_{1}...i_{n}}$,
$i_{j}\in\{1,...,N\}$,%
\begin{align*}
\int_{\lbrack0,1)\times X}\mathbf{1}_{A}\left(  \pi_{2}\circ S^{n}\right)
d\left(  \lambda\otimes\nu\right)   &  =\sum_{i_{1},...,i_{n}}\int
_{J_{i_{1}...i_{n}}}\mathbf{1}_{A}\left(  \pi_{2}\circ S^{n}\right)  d\left(
\lambda\otimes\nu\right) \\
&  =\sum_{i_{1},...,i_{n}}\int_{J_{i_{1}...i_{n}}}\mathbf{1}_{A}\left(
\tau_{i_{1}}\circ...\circ\tau_{i_{n}}\right)  d\left(  \lambda\otimes
\nu\right) \\
&  =\sum_{i_{1},...,i_{n}}\int_{[0,1)\times X}\mathbf{1}_{J_{i_{1}...i_{n}}%
}\mathbf{1}_{A}\left(  \tau_{i_{1}}\circ...\circ\tau_{i_{n}}\right)  d\left(
\lambda\otimes\nu\right)  ,
\end{align*}
where, in the second line, we have used that $\mathbf{1}_{A}\left(  \pi
_{2}\circ S^{n}\right)  =\mathbf{1}_{A}\left(  \tau_{i_{1}}\circ...\circ
\tau_{i_{n}}\right)  $ on the set $J_{i_{1}...i_{n}}$. If now we apply
Fubini's theorem,%
\[
\int_{\lbrack0,1)\times X}\mathbf{1}_{J_{i_{1}...i_{n}}}\mathbf{1}_{A}\left(
\tau_{i_{1}}\circ...\circ\tau_{i_{n}}\right)  d\left(  \lambda\otimes
\nu\right)  =\int_{X}\mathbf{1}_{A}\left(  \tau_{i_{1}}\circ...\circ
\tau_{i_{n}}(x)\right)  dv\left(  x\right)  \int_{[0,1)}\mathbf{1}%
_{J_{i_{1}...i_{n}}}\left(  y,x\right)  d\lambda\left(  y\right)  ,
\]
but by Lemma \ref{lemma slices go to zero}, for any fixed $x\in X$,%
\[
\int_{\lbrack0,1)}\mathbf{1}_{J_{i_{1}...i_{n}}}\left(  y,x\right)
d\lambda\left(  y\right)  =p_{i_{n}}(x)\cdots p_{i_{1}}(\tau_{i_{2}}%
\circ\cdots\circ\tau_{i_{n}}(x)).
\]
Therefore,%
\[
\int_{\lbrack0,1)\times X}\mathbf{1}_{A}\left(  \pi_{2}\circ S^{n}\right)
d\left(  \lambda\otimes\nu\right)  =\sum_{i_{1},...,i_{n}}\int_{X}p_{i_{n}%
}\left(  x\right)  \cdots p_{i_{1}}\left(  \tau_{i_{2}}\circ...\circ
\tau_{i_{n}}(x)\right)  \mathbf{1}_{A}\left(  \tau_{i_{1}}\circ...\circ
\tau_{i_{n}}(x)\right)  dv\left(  x\right)  .
\]
\smallskip
\end{proof}

\begin{corollary}
\label{corollary koopman}Let $U_{S}:L^{p}\left(  [0,1)\times X,\lambda
\otimes\mu\right)  \rightarrow L^{p}\left(  [0,1)\times X,\lambda\otimes
\mu\right)  $, $p\geq1$, be the \textbf{Koopman operator} associated to $S$,
i.e., $U_{S}f=f\circ S$. Let $\pi_{2}:[0,1)\times X\rightarrow X$ be the
projection onto the second factor. If $\nu$ is a measure on $\left(
X,\mathcal{B}\right)  $, then%
\[
\nu^{(n)}\left(  A\right)  =\int_{[0,1)\times\lbrack0,\pi]}U_{S}^{n}\left(
1_{\pi_{2}^{-1}(A)}\right)  d\left(  \lambda\otimes\nu\right)  ,~A\in
\mathcal{B},~n\in\mathbb{N}.
\]
\smallskip
\end{corollary}

Let $S$ be the skew-type representation of the random billiard introduced in
Section \ref{section model} and let $\nu\ll\mu$ be a probability measure
absolutely continuous with respect to the Liouville measure $\mu$ with
Radon-Nikodym derivative $f\in L^{1}\left(  [0,\pi],\mu\right)  $. One can
easily check that $\frac{d\left(  \lambda\otimes\nu\right)  }{d\left(
\lambda\otimes\mu\right)  }=\pi_{2}^{\ast}(f)$, where $\pi_{2}^{\ast
}(f)\left(  \omega\right)  :=f\left(  \pi_{2}\left(  \omega\right)  \right)  $
and $\omega\in\lbrack0,1)\times\lbrack0,\pi]$. If $S$ is mixing (or exact)
then, from Corollary \ref{corollary koopman},
\begin{equation}
\int\pi_{2}^{\ast}\left(  f\right)  U_{S}^{n}\left(  1_{\pi_{2}^{-1}%
(A)}\right)  d\left(  \lambda\otimes\mu\right)  \underset{n\rightarrow\infty
}{\longrightarrow}\int\pi_{2}^{\ast}\left(  f\right)  d\left(  \lambda
\otimes\mu\right)  \int\mathbf{1}_{\pi_{2}^{-1}\left(  A\right)  }d\left(
\lambda\otimes\mu\right)  =\mu\left(  A\right)  \label{eq 17}%
\end{equation}
(\cite[Proposition 4.4.1 (b)]{Lasota}), which implies the strong Knudsen's
law. We will prove that $S$ is exact in Section \ref{section exactness of S}
and give more details about the strong law in Section
\ref{section knudsen's law}. The proof uses the fact that the pull-back
$\pi_{2}^{\ast}\left(  \mathbf{1}_{C}\right)  $ of a characteristic function
$\mathbf{1}_{C}$, $C\in\mathcal{B}\left(  \left[  0,\pi\right]  \right)  $,
cannot be invariant by $S^{2}:=S\circ S$ (Section
\ref{section ergodicity of SS}). It is worth observing that the strong law,
unlike the standard approach to random billiards available in the literature
from the first return map (see \ref{eq 18}), is a consequence of the
asymptotic properties of the skew-type representation $S$.

\section{Properties of $S^{2}$\label{section ergodicity of SS}}

Let $S$ be the skew-type representation of the random map (\ref{eq tau maps})
and (\ref{eq probabilities model}). In this section, we are going to show that
$S^{2}=S\circ S:[0,1)\times\lbrack0,\pi]\rightarrow\lbrack0,1)\times
\lbrack0,\pi]$ cannot leave invariant any characteristic function $\pi
_{2}^{\ast}\left(  \mathbf{1}_{C}\right)  $, where $C\in\mathcal{B}\left(
\left[  0,\pi\right]  \right)  $ has probability $0<\mu\left(  C\right)  <1$
and $\pi_{2}:[0,1)\times X\rightarrow X$ denotes the projection onto the
second factor. This result will be used in the next section to prove that $S$
is exact. First, we will show that our model exhibits a very useful symmetry.

\begin{proposition}
\label{prop symmetry}Let $S$ be the skew-type representation of the random map
(\ref{eq tau maps}) and (\ref{eq probabilities model}). Let $g:[0,\pi
]\rightarrow\mathbb{R}$ be a function such that $\pi_{2}^{\ast}(g)$ is $S^{2}%
$-invariant, i.e.,%
\begin{equation}
\pi_{2}^{\ast}(g)=\pi_{2}^{\ast}(g)\circ S^{2}, \label{eq 13}%
\end{equation}
and let $\phi:[0,\pi]\rightarrow\lbrack0,\pi]$ be defined by $\theta\mapsto
\pi-\theta$. Then $\pi_{2}^{\ast}(g\circ\phi)$ is also $S^{2}$-invariant.
\end{proposition}

\begin{proof}
It is not difficult to realise from (\ref{eq probabilities model}) that%
\[
p_{1}\circ\phi=p_{3},~~p_{3}\circ\phi=p_{1},~\text{~}p_{2}\circ\phi
=p_{4},\text{ and }p_{4}\circ\phi=p_{2}%
\]
just using that the the function $u_{\alpha}(\theta)$ introduced in
(\ref{eq definition function u}) has period $\pi$. Furthermore,
straightforward computations shows that%
\[
\phi\circ\tau_{1}\circ\phi=\tau_{3},~~\phi\circ\tau_{3}\circ\phi=\tau
_{1},~~\phi\circ\tau_{2}\circ\phi=\tau_{4},\text{ and }\phi\circ\tau_{4}%
\circ\phi=\tau_{2}.
\]
Given an index $k\in\{1,...,4\}$, we define its conjugate index $\widetilde
{k}$ in the following manner:%
\[
\widetilde{1}=3,~~\widetilde{3}=1,~~\widetilde{2}=4,\text{ and }\widetilde
{4}=2.
\]

On the other hand, (\ref{eq 13}) is equivalent to%
\begin{equation}
g\left(  x\right)  =g\left(  \tau_{j}\left(  \tau_{i}(x)\right)  \right)
\text{ if }p_{i}\left(  x\right)  p_{j}\left(  \tau_{i}\left(  x\right)
\right)  >0. \label{eq 34}%
\end{equation}
We have to check that this property holds for $g\circ\phi$. So let
$x\in\lbrack0,\pi]$ such that $p_{i}\left(  x\right)  p_{j}\left(  \tau
_{i}\left(  x\right)  \right)  >0$. Since $p_{i}=p_{\widetilde{i}}\circ\phi$
and $\phi\circ\tau_{_{\widetilde{i}}}=\tau_{i}\circ\phi$ for any $i=1,...,4$,
we have that $p_{i}\left(  x\right)  p_{j}\left(  \tau_{i}\left(  x\right)
\right)  >0$ implies%
\[
\left.
\begin{array}
[c]{r}%
p_{i}\left(  x\right)  =p_{\widetilde{i}}\circ\phi\left(  x\right)  >0\\
p_{j}\left(  \tau_{i}\left(  x\right)  \right)  =p_{\widetilde{j}}\circ
\phi\circ\tau_{i}\left(  x\right)  =p_{\widetilde{j}}\circ\tau_{_{\widetilde
{j}}}\circ\phi\left(  x\right)  >0
\end{array}
\right\}  \Longrightarrow(p_{\widetilde{i}}\cdot(p_{\widetilde{j}}\circ
\tau_{_{\widetilde{j}}}))\circ\phi\left(  x\right)  >0.
\]
Therefore, using the invariance of $g$ expressed in (\ref{eq 34})%
\[
g\left(  \phi(x)\right)  =g\left(  \tau_{\widetilde{j}}\circ\tau
_{\widetilde{i}}\circ\phi(x)\right)  =g\left(  \tau_{\widetilde{j}}\circ
\phi\circ\tau_{i}(x)\right)  =\left(  g\circ\phi\right)  \left(  \tau_{j}%
\circ\tau_{i}(x)\right)  ,
\]
so $g\circ\phi$ also satisfies (\ref{eq 34}) and is $S^{2}$-invariant.
\end{proof}

\begin{proposition}
\label{prop g invariant S2 and phi}Let $g\in L^{1}\left(  [0,\pi],\mu\right)
$ and suppose that $\pi_{2}^{\ast}(g)$ is invariant by $S^{2}$ and $\phi$. If
$\alpha$ in (\ref{eq probabilities model}) and (\ref{eq tau maps}) is
irrational, then $g$ is constant.
\end{proposition}

\begin{proof}
Let $\left.  [0,\pi]\right/  _{0\sim\pi}$ be the quotient space obtained by
identifying $0\sim\pi$. This space is homeomorphic to the unit circle
$\mathbb{S}^{1}$. The function $g:[0,\pi]\rightarrow\mathbb{R}$ induces a
function $[g]:\left.  [0,\pi]\right/  _{0\sim\pi}-\{[0]\}\rightarrow
\mathbb{R}$ on the quotient that is well defined except at $[0]$, the
equivalence class of $0$, because maybe $g(0)\neq g\left(  \pi\right)  $. We
are going to prove that $[g]$ is invariant by an irrational rotation. Since
constants are the only integrable functions invariant by irrational rotations,
$[g]$ must be constant and $g$ too.

Recall that $\pi_{2}^{\ast}\left(  g\right)  =\pi_{2}^{\ast}\left(  g\right)
\circ S^{2}$ reads as%
\begin{equation}
g\left(  x\right)  =g\left(  \tau_{j}\left(  \tau_{i}(x)\right)  \right)
\text{ if }p_{i}\left(  x\right)  p_{j}\left(  \tau_{i}\left(  x\right)
\right)  >0. \tag{\ref{eq 34}}%
\end{equation}
We are going split the interval $\left[  0,\pi\right]  $ into subintervals
and, for each of them, we will explore the invariance properties of $g$
implied by (\ref{eq 34}).

\begin{description}
\item[(i)] If $x\in(0,\pi-4\alpha)$, then $p_{1}\left(  x\right)  p_{1}\left(
\tau_{1}\left(  x\right)  \right)  =p_{1}\left(  x\right)  p_{1}\left(
x+2\alpha\right)  >0$ so%
\[
g\left(  x\right)  =g\left(  \tau_{1}\left(  \tau_{1}(x)\right)  \right)
=g\left(  x+4\alpha\right)  .
\]

\item[(ii)] If $x\in\lbrack\pi-4\alpha,\pi-3\alpha)$, then $p_{1}\left(
x\right)  p_{2}\left(  \tau_{1}\left(  x\right)  \right)  =p_{1}\left(
x\right)  p_{2}\left(  x+2\alpha\right)  >0$. Writing $x=\pi-4\alpha+z$ with
$z\in\lbrack0,\alpha)$,%
\[
g\left(  \pi-4\alpha+z\right)  =g\left(  \tau_{2}\left(  \tau_{1}(\pi
-4\alpha+z)\right)  \right)  =g\left(  \pi-2\alpha-z\right)  .
\]
Now, since $g$ is $\phi$-invariant,%
\[
g\left(  \pi-2\alpha-z\right)  =g\left(  \phi\left(  \pi-2\alpha-z\right)
\right)  =g\left(  2\alpha+z\right)
\]
Therefore,%
\[
\left[  g\right]  \left(  \left[  x\right]  \right)  =\left[  g\right]
\left(  \left[  x+6\alpha\right]  \right)  \text{ if }x\in\lbrack\pi
-4\alpha,\pi-3\alpha).
\]

\item[(iii)] If $x\in\lbrack\pi-3\alpha,\pi-2\alpha)$ then $p_{2}\left(
x\right)  p_{3}\left(  \tau_{2}\left(  x\right)  \right)  >0$. Writing
$x=\pi-2\alpha-z$ with $z\in\lbrack0,\alpha)$,%
\[
g\left(  \pi-2\alpha-z\right)  =g\left(  \tau_{3}\left(  \tau_{2}(\pi
-2\alpha-z)\right)  \right)  =g\left(  \pi-4\alpha+z\right)  .
\]
Using $g=g\circ\phi$,%
\[
g\left(  \pi-4\alpha+z\right)  =g\left(  \phi\left(  \pi-4\alpha+z\right)
\right)  =g\left(  4\alpha-z\right)  .
\]
That is,%
\[
\left[  g\right]  \left(  \left[  x\right]  \right)  =\left[  g\right]
\left(  \left[  x+6\alpha\right]  \right)  \text{ if }x\in\lbrack\pi
-3\alpha,\pi-2\alpha).
\]

\item[(iv)] Suppose $x\in\lbrack\pi-2\alpha,\pi-\alpha)$, $x=\pi-2\alpha+z$
with $z\in\lbrack0,\alpha)$. Since $p_{2}\left(  x\right)  p_{1}\left(
\tau_{2}\left(  x\right)  \right)  >0$ on $[\pi-2\alpha,\pi-\alpha)$, by
(\ref{eq 34}),%
\[
g\left(  x\right)  =g\left(  \tau_{2}\circ\tau_{1}(x)\right)  ,~~\tau_{2}%
\circ\tau_{1}(x)=\pi-z.
\]
Since $g$ is invariant by $\phi$,%
\[
g(\pi-z)=g\left(  \phi\left(  \pi-z\right)  \right)  =g\left(  z\right)
,~~z\in\lbrack0,\alpha).
\]
Summarising, $g\left(  \pi-2\alpha+z\right)  =g(z)$ which implies%
\[
\lbrack g]\left(  [x]\right)  =[g]\left(  [x+2\alpha]\right)  \text{ if }%
x\in\lbrack\pi-2\alpha,\pi-\alpha).
\]
Now $z\in\lbrack0,\alpha)$ and we can use the $S^{2}$-invariance expressed in
the first item so that%
\[
\left[  g\right]  \left(  \left[  x\right]  \right)  =\left[  g\right]
\left(  \left[  x+2\alpha\right]  \right)  =\left[  g\right]  \left(  \left[
x+6\alpha\right]  \right)  .
\]

\item[(v)] Finally, let $x\in\lbrack\pi-\alpha,\pi)$, $x=\pi-\alpha+z$ with
$z\in\lbrack0,\alpha)$. Since $p_{3}\left(  x\right)  p_{2}\left(  \tau
_{3}\left(  x\right)  \right)  >0$ on $[\pi-\alpha,\pi)$%
\[
g\left(  x\right)  =g\left(  \tau_{2}\circ\tau_{3}(x)\right)  ,~~\tau_{2}%
\circ\tau_{3}(x)=\pi-\alpha-z.
\]
by (\ref{eq 34}). Using the $\phi$-invariance of $g$,%
\[
g(\pi-\alpha-z)=g\left(  \phi\left(  \pi-\alpha-z\right)  \right)  =g\left(
z+\alpha\right)  .
\]
That is, $g\left(  x\right)  =g(z+\alpha)$, which again implies%
\[
\lbrack g]\left(  [x]\right)  =[g]\left(  [x+2\alpha]\right)  \text{ if }%
x\in\lbrack\pi-\alpha,\pi).
\]
Now $x+2\alpha\in\lbrack\alpha,2\alpha)$ and $2\alpha<\pi-4\alpha$ as by
assumption $\alpha<\pi/6$; we are in the situation of (i) so%
\[
\left[  g\right]  \left(  \left[  x\right]  \right)  =\left[  g\right]
\left(  \left[  x+2\alpha\right]  \right)  =\left[  g\right]  \left(  \left[
x+6\alpha\right]  \right)  .
\]

\end{description}

In conclusion, if $g$ is invariant by $S^{2}$ and $\phi$, then%
\begin{equation}
\left[  g\right]  \left(  \left[  x\right]  \right)  =\left\{
\begin{array}
[c]{l}%
\left[  g\right]  \left(  \left[  x+4\alpha\right]  \right)  \text{ \ if }%
x\in(0,\pi-4\alpha)\\
\left[  g\right]  \left(  \left[  x+6\alpha\right]  \right)  \text{ \ if }%
x\in\lbrack\pi-4\alpha,\pi).
\end{array}
\right.  \label{eq 31}%
\end{equation}
Let $k\in\mathbb{N}$ such that $\left(  4k+6\right)  \alpha>\pi>\left(
4\left(  k-1\right)  +6\right)  \alpha$. Observe that $k\geq1$ because
$\alpha<\pi/6$. Applying the invariance expressed in (\ref{eq 31})
iteratively, we have%
\begin{equation}
\left[  g\right]  \left(  \left[  x\right]  \right)  =\left[  g\right]
\left(  \left[  x+\left(  4k+6\right)  \alpha\right]  \right)  \text{ for any
}x\in\left(  0,\pi\right)  . \label{eq 39}%
\end{equation}
In other words, $\left[  g\right]  $ is invariant by the rotation of angle
$\left(  4k+6\right)  \alpha$. If $\alpha/2\pi$ is irrational, $g$ must be constant.
\end{proof}

\begin{proposition}
\label{propositon characteristic function invariant}Let $C\in\mathcal{B}%
\left(  [0,\pi]\right)  $ and suppose that $\pi_{2}^{\ast}(\mathbf{1}_{C})$ is
$S^{2}$-invariant. If $\alpha$ in (\ref{eq probabilities model}) and
(\ref{eq tau maps}) is irrational, then $\mu\left(  C\right)  $ equals either
$1$ or $0$.
\end{proposition}

\begin{proof}
Suppose that $0<\mu\left(  C\right)  <1$. We want to prove that $\pi_{2}%
^{\ast}(\mathbf{1}_{C})$ cannot be $S^{2}$-invariant. Since constants are
trivially $S^{2}$-invariant, we can subtract $\mu\left(  C\right)  $ from
$\mathbf{1}_{C}$ so that $\mathbf{1}_{C}-\mu\left(  C\right)  $ is still
$S^{2}$-invariant and has expectation $0$. Let $g:=\mathbf{1}_{C}-\mu\left(
C\right)  $ and define $\overline{g}:=\frac{1}{2}\left(  g+g\circ\phi\right)
$, which is clearly $\phi$-invariant because $\phi^{2}=\operatorname*{Id}$.
Since $\pi_{2}^{\ast}(g)$ is $S$-invariant, so are $\pi_{2}^{\ast}(g\circ
\phi)$ (Proposition \ref{prop symmetry}) and $\pi_{2}^{\ast}(\overline{g})$.
By Proposition \ref{prop g invariant S2 and phi}, $\overline{g}$ is constant
and equal to its expectation $\operatorname*{E}_{\mu}\left[  \overline
{g}\right]  $. But $\operatorname*{E}\nolimits_{\mu}\left[  g\circ\phi\right]
=\operatorname*{E}\nolimits_{\mu}\left[  g\right]  =0$ because $d\mu
=\sin\left(  x\right)  dx$ and $\sin\left(  x\right)  $ are invariant by
$\phi$. Therefore, $\operatorname*{E}_{\mu}\left[  \overline{g}\right]  =0$
and $\overline{g}$, which is constant, must be equal to $0$. That is,%
\begin{equation}
g=-g\circ\phi. \label{eq 25}%
\end{equation}

Since%
\[
g\left(  x\right)  =\mathbf{1}_{C}\left(  x\right)  -\mu\left(  C\right)
=\left\{
\begin{array}
[c]{l}%
1-\mu\left(  C\right)  >0\text{ if }x\in C\\
-\mu\left(  C\right)  <0\text{ if }x\notin C,
\end{array}
\right.
\]
(remember that we assumed $0<\mu\left(  C\right)  <1$), we conclude from
(\ref{eq 25}) that $1-\mu\left(  C\right)  =-\left(  -\mu\left(  C\right)
\right)  \Longrightarrow\mu\left(  C\right)  =1/2$ and%
\[%
\begin{array}
[c]{l}%
g\left(  x\right)  =1/2\Longleftrightarrow x\in C,\\
g(x)=-1/2\Longleftrightarrow x\in C^{c}.
\end{array}
\]
Moreover, looking carefully at the proof of Proposition
\ref{prop g invariant S2 and phi}, we have%
\[
\left[  g\right]  \left(  \left[  x\right]  \right)  =\left\{
\begin{array}
[c]{l}%
\left[  g\right]  \left(  \left[  x+4\alpha\right]  \right)  \text{ \ if }%
x\in(0,\pi-4\alpha)\\
-\left[  g\right]  \left(  \left[  x+6\alpha\right]  \right)  \text{ \ if
}x\in\lbrack\pi-4\alpha,\pi),
\end{array}
\right.
\]
where instead of $g=g\circ\phi$ we have now used $g=-g\circ\phi$. Therefore,
(see (\ref{eq 39}))%
\begin{equation}
\left[  g\right]  \left(  \left[  x\right]  \right)  =-\left[  g\right]
\left(  \left[  x+\left(  4k+6\right)  \alpha\right]  \right)  \text{ for any
}x\in(0,\pi), \label{eq 40}%
\end{equation}
where $k\geq1$ is such that $\left(  4k+6\right)  \alpha>\pi>\left(  4\left(
k-1\right)  +6\right)  \alpha$. The minus sign in (\ref{eq 40}) tells us that
the rotation $R_{\beta}$ of angle $\beta:=\left(  4k+6\right)  \alpha-\pi$
sends $C$ to $C^{c}$ and vice versa.

Let $\mathbf{1}_{C}=\sum_{n=-\infty}^{\infty}a_{n}\operatorname*{e}%
\nolimits^{2inx}$ and $\mathbf{1}_{C^{c}}=\sum_{n=-\infty}^{\infty}%
\widetilde{a}_{n}\operatorname*{e}\nolimits^{2inx}$ the Fourier expansions of
$\mathbf{1}_{C}$ and $\mathbf{1}_{C^{c}}$ respectively. Since $\mathbf{1}%
_{C}+\mathbf{1}_{C^{c}}=\mathbf{1}_{[0,\pi]}$, we have that%
\begin{gather*}
\widetilde{a}_{n}=-a_{n}\text{ if }n\neq0\\
a_{0}+\widetilde{a}_{0}=1.
\end{gather*}
On the other hand, we already argued that $\mathbf{1}_{C}\circ R_{\beta
}=\mathbf{1}_{C^{c}}$. Imposing that the Fourier coefficients of
$\mathbf{1}_{C}$ and $\mathbf{1}_{C^{c}}$ are unique, we deduce that%
\[
a_{n}\operatorname*{e}\nolimits^{2in\beta}=\widetilde{a}_{n}=-a_{n}\text{ for
any }n\neq0.
\]
As $\beta/2\pi$ is irrational, $\operatorname*{e}\nolimits^{2in\beta}\neq-1$
for any $n\neq0$, which implies that $a_{n}=0$ for any $n\neq0$ so
$\mathbf{1}_{C}$ is constant a.s.. But this is clearly a contradiction because
(\ref{eq 25}) implied $\mu\left(  C\right)  =1/2$.
\end{proof}

\section{Exactness of $S$\label{section exactness of S}}

In this section, we are going to prove that $S$ is exact. This will imply
Knudsen's strong law for the random billiard introduced in Section
\ref{section model}.

Let $T:\Omega\rightarrow\Omega$ be a measurable transformation of a
probability space $\left(  \Omega,\mathcal{F},\mu\right)  $. From the
measurability of $T$, we have the chain of $\sigma$-algebras%
\begin{equation}
...\subseteq T^{-n}\left(  \mathcal{F}\right)  \subseteq...\subseteq
T^{-2}\left(  \mathcal{F}\right)  \subseteq T^{-1}\left(  \mathcal{F}\right)
\subseteq\mathcal{F}. \label{eq 32}%
\end{equation}
The map $T$ is called {\bfseries\itshape exact} if the $\sigma$-algebra
$\mathcal{G}:=\bigcap\nolimits_{n\geq0}T^{-n}\left(  \mathcal{F}\right)  $
only contains sets of measure either $0$ or $1$. It is not difficult to prove
that sets in $\mathcal{G}$ are characterised by the property
\[
A\in\mathcal{G}\Longleftrightarrow A=T^{-n}\left(  T^{n}\left(  A\right)
\right)  \text{ for any }n\in\mathbb{N}.
\]
This characterisation, in turn, implies that a map $T$ is exact if and only if%
\[
\lim_{n\rightarrow\infty}\mu\left(  T^{n}\left(  A\right)  \right)  =1
\]
for any $A\in\mathcal{\mathcal{F}}$ such that $\mu\left(  A\right)  >0$
(\cite[Section 2.2]{rohlin}). Equation (\ref{eq 32}) reads at the level of
$L^{2}$ spaces as%
\[
...\subseteq U_{T}^{n}L^{2}\subseteq...\subseteq U_{T}^{2}L^{2}\subseteq
U_{T}L^{2}\subseteq L^{2}%
\]
where $L^{2}=L_{\mathbb{C}}^{2}\left(  \Omega,\mathcal{F},\mu\right)  $ and
$U_{T}\left(  f\right)  =f\circ T$ is the {\bfseries\itshape Koopman operator}
defined on $L^{2}$. If a map is exact, then%
\[
\bigcap\nolimits_{n=0}^{\infty}U_{T}^{n}L^{2}=\mathbb{C},
\]
which means that $U_{T}^{n}\left(  f\right)  \rightarrow\operatorname*{E}%
_{\mu}\left[  f\right]  $ in $L^{2}$ as $n\rightarrow\infty$ (see
\cite[Section 2.5]{rohlin}). If $\nu$ is a probability absolutely continuous
with respect to $\mu$ with Radon-Nikodym derivative $f\in L^{1}$, then%
\[
T_{\ast}^{n}\nu\left(  A\right)  =\operatorname*{E}\nolimits_{\mu}\left[
fU_{T}^{n}\left(  \mathbf{1}_{A}\right)  \right]  \underset{n\rightarrow
\infty}{\longrightarrow}\mu\left(  A\right)  \operatorname*{E}\nolimits_{\mu
}\left[  f\right]  =\mu(A),
\]
that is, the sequence of measures $T_{\ast}^{n}\nu$ converges (weakly) to
$\mu$ for any absolutely continuous measure $\nu\ll\mu$ (\cite[Proposition
4.4.1 (b)]{Lasota}, \cite[Section 2.6]{rohlin}).

We want to show that the skew-type representation $S:[0,1)\times\lbrack
0,\pi]\rightarrow\lbrack0,1)\times\lbrack0,\pi]$ of the model introduced in
Section \ref{section model} is exact with respect to the Liouville measure to
conclude that, for any absolutely continuous measure $\nu\ll\mu$ on
$\mathcal{B}\left(  [0,\pi]\right)  $ and any $A\in\mathcal{B(}[0,\pi])$,
\begin{equation}
\nu^{(n)}\left(  A\right)  =\int_{[0,\pi]}\mathbf{1}_{A}\left(  \pi_{2}\circ
S^{n}\right)  d\nu=\int_{[0,\pi]}U_{S}^{n}\left(  \mathbf{1}_{\pi_{2}^{-1}%
(A)}\right)  d\lambda\otimes\nu\underset{n\rightarrow\infty}{\longrightarrow
}\mu\left(  A\right)  , \label{eq 33}%
\end{equation}
i.e., that the strong Knudsen's law holds. Observe from (\ref{eq 33}) that we
only need to consider the action of the Koopman operator $U_{S}$ on square
integrable functions that are the pull-back by $\pi_{2}$ of functions in
$L_{\mathbb{C}}^{2}\left(  [0,\pi],\mathcal{B}\left(  [0,\pi]\right)
,\mu\right)  $. We are going to denote $L_{\mathbb{C}}^{2}\left(
[0,\pi],\mathcal{B}\left(  [0,\pi]\right)  ,\mu\right)  $ simply by
$L^{2}\left(  [0,\pi]\right)  $ for the sake of a clearer notation.
Consequently, we do not need to check that $S$ is exact for the Borel $\sigma
$-algebra of $[0,1)\times\lbrack0,\pi]$ but it is enough to consider a smaller
one, namely, the smallest $\sigma$-algebra that makes both the functions in
$\pi_{2}^{\ast}\left(  L^{2}\left(  [0,\pi]\right)  \right)  $ and $S$ measurable.

\begin{definition}
For any $n\in\mathbb{N}$, let
\[
\mathcal{F}_{n}:=\sigma\left(  \left\{  \pi_{2}^{-1}\left(  U\right)  \bigcap
J_{i_{1}...i_{n}}~%
\vert
~U\in\mathcal{B}\left(  [0,\pi]\right)  \right\}  \right)  ,
\]
where $J_{i_{1}...i_{n}}$ are as in Definition \ref{definition slices}. For
any $U\in\mathcal{B}\left(  [0,\pi]\right)  $, the sets $\pi_{2}^{-1}\left(
U\right)  \bigcap J_{i_{1}...i_{n}}$ and $\pi_{2}^{-1}\left(  U\right)  $ will
be called \textbf{generators} (of $\mathcal{F}_{n}$).
\end{definition}

The sequence of $\sigma$-algebras $\left\{  \mathcal{F}_{n}\right\}
_{n\in\mathbb{N}}$ define a filtration,%
\[
\mathcal{F}_{1}\subseteq\mathcal{F}_{2}\subseteq...\subseteq\mathcal{F}%
_{n}\subseteq...
\]
Indeed, for any $U\in\mathcal{B}\left(  [0,\pi]\right)  $ and any $n\geq1$, a
generator $\pi_{2}^{-1}\left(  U\right)  \bigcap J_{i_{1}...i_{n}}$ can be
written as%
\[
\pi_{2}^{-1}\left(  U\right)  \bigcap J_{i_{1}...i_{n}}=\bigcup_{k=1}%
^{N}\left(  \pi_{2}^{-1}\left(  U\right)  \bigcap J_{ki_{1}...i_{n}}\right)
,
\]
which implies $\mathcal{F}_{n}\subseteq\mathcal{F}_{n+1}$. We define
$\mathcal{F}$ as the limit of this filtration.

\begin{definition}%
\[
\mathcal{F}:=\sigma\left(  \bigcup\nolimits_{n\geq1}\mathcal{F}_{n}\right)  .
\]

\end{definition}

\begin{proposition}
$S:[0,1)\times\lbrack0,\pi]\rightarrow\lbrack0,1)\times\lbrack0,\pi]$ is
$\mathcal{F}$-measurable. For any $f\in L^{2}\left(  [0,\pi]\right)  $,
$\pi_{2}^{\ast}\left(  f\right)  $ is also $\mathcal{F}$-measurable.
\end{proposition}

\begin{proof}
The second statement of the proposition is obvious. To prove that $S$ is
$\mathcal{F}$-measurable it is enough to show that $S^{-1}\left(  A\right)
\in\mathcal{F}$ for the generators of $\mathcal{F}$. So let $A\in\mathcal{F}$
such that $A=\pi_{2}^{-1}\left(  U\right)  \bigcap J_{i_{1}...i_{n}}$ where
$U\in\mathcal{B}\left(  [0,\pi]\right)  $ and $n\geq1$. Then, by definition of
the sets $J_{i_{1}...i_{n}}$%
\[
S^{-1}\left(  A\right)  =\bigcup_{k=1}^{N}\left(  \pi_{2}^{-1}\left(  \tau
_{k}^{-1}\left(  U\right)  \right)  \bigcap J_{i_{1}...i_{n}k}\right)
\]
which clearly belongs to $\mathcal{F}$.\smallskip
\end{proof}

The following two lemmas aim at getting a better insight of the structure of
the $\sigma$-algebra $\mathcal{F}$. As an immediate consequence of them, it
will be enough to show that $S$ is exact by looking at its action on the
generators of $\mathcal{F}$.

\begin{lemma}
\label{prop finite union generators}Any set $A\in\mathcal{F}_{n}$ can be
expressed as a finite union of disjoint generators.
\end{lemma}

\begin{proof}
First to all, observe that, for a fixed $n\in\mathbb{N}$, the sets
$J_{i_{1}...i_{n}}$ are disjoint and form a partition of $[0,1)\times
\lbrack0,\pi]$.

The complement of a generator $\pi_{2}^{-1}\left(  U\right)  \bigcap
J_{i_{1}...i_{n}}$ is%
\[
\left(  \pi_{2}^{-1}\left(  U\right)  \bigcap J_{i_{1}...i_{n}}\right)
^{c}=\left(  \biguplus_{j_{1}\neq i_{1},...,j_{n}\neq i_{n}}\pi_{2}%
^{-1}\left(  U\right)  \bigcap J_{j_{1}...j_{n}}\right)  \biguplus\pi_{2}%
^{-1}\left(  \left[  0,\pi\right]  \backslash U\right)  ,
\]
where $\biguplus$ denotes the disjoint union, so it can be expressed as a
finite union of generators. On the other hand, if we take a countable union of
generators $\left\{  C_{k}\right\}  _{k\geq1}$, $C_{k}=\pi_{2}^{-1}\left(
U_{k}\right)  \bigcap J_{i_{1}^{k}...i_{n}^{k}}$,%
\[
\bigcup_{k\geq1}\left(  \pi_{2}^{-1}\left(  U_{k}\right)  \bigcap J_{i_{1}%
^{k}...i_{n}^{k}}\right)  =\biguplus_{i_{1}^{k},...,i_{n}^{k}}\left(
\bigcup\nolimits_{m\in I_{i_{1}^{k}...i_{n}^{k}}}\pi_{2}^{-1}\left(
U_{m}\right)  \right)  \bigcap J_{i_{1}^{k}...i_{n}^{k}}%
\]
where $I_{i_{1}^{k}...i_{n}^{k}}:=\{k\geq1~|~\pi_{2}^{-1}\left(  U_{k}\right)
\bigcap J_{i_{1}^{k}...i_{n}^{k}}$ is in $\left\{  C_{m}\right\}  _{m\geq1}%
\}$. But%
\[
\biguplus_{i_{1}^{k},...,i_{n}^{k}}\left(  \bigcup\nolimits_{m\in I_{i_{1}%
^{k}...i_{n}^{k}}}\pi_{2}^{-1}\left(  U_{m}\right)  \right)  \bigcap
J_{i_{1}^{k}...i_{n}^{k}}=\biguplus_{i_{1}^{k},...,i_{n}^{k}}\left(  \pi
_{2}^{-1}\left(  \bigcup\nolimits_{m\in I_{i_{1}^{k}...i_{n}^{k}}}%
U_{k}\right)  \right)  \bigcap J_{i_{1}^{k}...i_{n}^{k}}.
\]
where the indices $i_{1}^{k},...,i_{n}^{k}$ can only take finite number of
possibilities. Therefore, a countable union of generators reduces to a finite
union.\smallskip
\end{proof}

The union $\bigcup\nolimits_{n\geq1}\mathcal{F}_{n}$ of a filtration
$\mathcal{F}_{1}\subseteq\mathcal{F}_{2}\subseteq...$ is an algebra of sets.
The $\sigma$-algebra generated by an algebra of sets $\mathcal{A}$ is
characterised by containing all the sets in $\mathcal{A}$ and the limit of all
monotone sequence of sets. That is, $\sigma\left(  \mathcal{A}\right)  $
coincides with the monotone class generated by $\mathcal{A}$ (\cite[Section
1.3, Theorem 1]{chow probability theory}). This observation is the key to
proving the following Lemma:

\begin{lemma}
\label{Lemma A contains generator}If $A\in\mathcal{F=}\sigma\left(
\bigcup\nolimits_{n\geq1}\mathcal{F}_{n}\right)  $ has positive probability,
then $A$ contains a generator of positive probability.
\end{lemma}

\begin{proof}
Let $A\in\sigma(%
{\textstyle\bigcup\nolimits_{n\geq1}}
\mathcal{F}_{n})$. If $A\in\bigcup\nolimits_{n\geq1}\mathcal{F}_{n}$, then
$A\in\mathcal{F}_{n}$ for some $n\in\mathbb{N}$. By Proposition
\ref{prop finite union generators}, $A$ can be expressed as a finite union of
generators and at least one of them must have positive probability. If
$A\notin\bigcup\nolimits_{n\geq1}\mathcal{F}_{n}$, then there exists a
monotone sequence of sets $\left\{  B_{n}\right\}  _{n\in\mathbb{N}}%
\subset\bigcup\nolimits_{n\geq1}\mathcal{F}_{n}$ such that $A=\lim_{n}B_{n}$.
We are going to deal with the cases $\left\{  B_{n}\right\}  _{n\in\mathbb{N}%
}$ increasing or decreasing separately.

\begin{itemize}
\item If $\left\{  B_{n}\right\}  _{n\in\mathbb{N}}$ is increasing, then
$A=\bigcup\nolimits_{n\geq1}B_{n}$. Since $\lambda\otimes\mu\left(  A\right)
>0$, at least one $B_{n}$ must have positive measure, $\lambda\otimes
\mu\left(  B_{n}\right)  >0$. But $B_{n}\in\mathcal{F}_{m}$ for some $m$, so
it contains a generator of positive probability.

\item Let $\left\{  B_{n}\right\}  _{n\in\mathbb{N}}\subset\bigcup
\nolimits_{n\geq1}\mathcal{F}_{n}$ be a decreasing sequence such that
$A=\bigcap\nolimits_{n\geq1}B_{n}$. Suppose that $0<\mu\left(  A\right)  <1$
(otherwise there is nothing to prove). We have%
\[
A=\left(  \bigcup\nolimits_{n\geq1}B_{n}^{c}\right)  ^{c}=\left(  \left(
\biguplus\nolimits_{n\geq1}\left.  B_{n+1}^{c}\right\backslash B_{n}%
^{c}\right)  \biguplus B_{1}^{c}\right)  ^{c},
\]
so that $A$ can be expressed as the complement of a countable union of
disjoint sets. Rename $C_{1}:=B_{1}^{c}$ and $C_{n}:=\left.  B_{n+1}%
^{c}\right\backslash B_{n}^{c}$, and express any $C_{n}$ as a disjoint union
of generators (Lemma \ref{prop finite union generators}),%
\[
C_{n}=\biguplus_{r=1}^{k_{n}}\pi_{2}^{-1}\left(  U_{r_{n}}\right)  \bigcap
J_{i_{1}^{n}...i_{r}^{n}}.
\]
In this decomposition, we only consider sets of strictly positive probability.
Then,%
\begin{equation}
A^{c}=\biguplus_{n\geq1}\biguplus_{r_{n}=1}^{k_{n}}\pi_{2}^{-1}\left(
U_{r_{n}}\right)  \bigcap J_{i_{1}^{n}...i_{r}^{n}}=\biguplus_{i_{1}%
^{n}...i_{r}^{n}}\left(  \bigcup\nolimits_{m\in I_{i_{1}^{n}...i_{r}^{n}}}%
\pi_{2}^{-1}\left(  U_{m}\right)  \right)  \bigcap J_{i_{1}^{n}...i_{r}^{n}}
\label{eq 38}%
\end{equation}
where $I_{i_{1}^{n}...i_{r}^{n}}:=\{m\geq1~|~\pi_{2}^{-1}\left(  U_{m}\right)
\bigcap J_{i_{1}^{n}...i_{r}^{n}}\subseteq A^{c}=\biguplus\nolimits_{n}%
C_{n}\}$. We claim that there exists a generator $G$ of strictly positive
probability that does not intersect $A^{c}$, i.e., $G\subseteq A$.

Let $i_{1}^{n}...i_{r}^{n}$ be a sequence appearing in the decomposition
(\ref{eq 38}). If $\mu(\bigcup\nolimits_{m\in I_{i_{1}^{n}...i_{r}^{n}}}%
U_{m})<1$, then%
\[
\left(  \left.  \left[  0,\pi\right]  \right\backslash \pi_{2}^{-1}\left(
\bigcup\nolimits_{m\in I_{i_{1}^{n}...i_{r}^{n}}}U_{m}\right)  \right)
\bigcap J_{i_{1}^{n}...i_{r}^{n}}%
\]
is a generator of strictly positive measure that does not intersect $A^{c}$.
If $\mu(\bigcup\nolimits_{m\in I_{i_{1}^{n}...i_{r}^{n}}}U_{m})=1$ for any
finite sequence $i_{1}^{n}...i_{r}^{n}$ in (\ref{eq 38}), then $A^{c}%
=\biguplus_{i_{1}^{n}...i_{r}^{n}}J_{i_{1}^{n}...i_{r}^{n}}$ a.s. and since we
assumed that $0<\mu\left(  A^{c}\right)  <1$, there must exist some set
$J_{i_{1}...i_{k}}$ with positive probability that does not appear in the
decomposition of $A^{c}$. That is, $J_{i_{1}...i_{k}}=\pi_{2}^{-1}\left(
[0,\pi]\right)  \bigcap J_{i_{1}...i_{k}}\subseteq A$.
\end{itemize}
\end{proof}

\begin{theorem}
For any $A\in\mathcal{F}$,
\[
\lim_{n\rightarrow\infty}\lambda\otimes\mu\left(  S^{n}\left(  A\right)
\right)  =1.
\]
In other words, $S:[0,1)\times\lbrack0,\pi]\rightarrow\lbrack0,1)\times
\lbrack0,\pi]$ is an exact endomorphism of $\left(  [0,1)\times\lbrack
0,\pi],\mathcal{F},\lambda\otimes\mu\right)  $.
\end{theorem}

\begin{proof}
Let $A\in\mathcal{F}$ be such that $\lambda\otimes\mu\left(  A\right)  >0$. By
Lemma \ref{Lemma A contains generator}, $A$ contains a generator $G=\pi
_{2}^{-1}\left(  U\right)  \bigcap J_{i_{1}...i_{n}}$ of positive measure for
some $U\in\mathcal{B}\left(  [0,\pi]\right)  $, $n\geq0$. It is an immediate
consequence of the definitions that, after $n$ iterations, the set
$S^{n}\left(  G\right)  $ has all its fibers of length $1$, that is,%
\[
\lambda\left(  \left(  \lbrack0,1)\times\{x\}\right)  \bigcap S^{n}\left(
G\right)  \right)  =1\text{ for any }x=\tau_{i_{1}}\circ...\circ\tau_{i_{n}%
}\left(  u\right)  ,~u\in U.
\]

Let $B:=S^{n}\left(  G\right)  $ and let $I_{x}:=[0,1)\times\{x\}$ be the
fiber at $x\in\lbrack0,\pi]$. Looking carefully at (\ref{eq tau maps}) and
(\ref{eq probabilities model}), one can see that%
\[
\tau_{2}\circ\tau_{2}=\tau_{4}\circ\tau_{4}=\tau_{1}\circ\tau_{3}=\tau
_{3}\circ\tau_{1}=\operatorname*{Id},
\]
where $\operatorname*{Id}$ denotes the identity on $[0,\pi]$, and%
\begin{align*}
p_{2}\left(  x\right)   &  >0\Rightarrow p_{2}\left(  x\right)  p_{2}\left(
\tau_{2}\left(  x\right)  \right)  >0\\
p_{4}\left(  x\right)   &  >0\Rightarrow p_{4}\left(  x\right)  p_{4}\left(
\tau_{4}\left(  x\right)  \right)  >0\\
p_{1}\left(  x\right)   &  >0\Rightarrow p_{1}\left(  x\right)  p_{3}\left(
\tau_{1}\left(  x\right)  \right)  >0\\
p_{3}\left(  x\right)   &  >0\Rightarrow p_{3}\left(  x\right)  p_{1}\left(
\tau_{3}\left(  x\right)  \right)  >0.
\end{align*}
These remarks imply that, if $I_{x}\cap J_{i}\neq\emptyset$,
\[
I_{x}\cap J_{i}\subseteq S^{2}\left(  I_{x}\cap J_{i}\right)  .
\]
Since all the fibers of $B$ have length $1$, $B\subseteq S^{2}\left(
B\right)  $ and, iteratively,%
\[
B\subseteq S^{2}\left(  B\right)  \subseteq...\subseteq S^{2n}\left(
B\right)  \subseteq...
\]
Define $C:=\bigcup\nolimits_{n\geq0}S^{2n}\left(  B\right)  $. Obviously
$S^{2}\left(  C\right)  =C$, the fibers of $C$ have all length $1$, and
$\lambda\otimes\mu\left(  C\right)  =\lim_{n\rightarrow\infty}\lambda
\otimes\mu\left(  S^{2n}\left(  B\right)  \right)  $. We want to show that
$\lambda\otimes\mu\left(  C\right)  =1$. Observe that $\lambda\otimes
\mu\left(  C\right)  >0$ because $\lambda\otimes\mu\left(  B\right)  >0$. From
$S^{2}\left(  C\right)  =C$ we have%
\[
C\subseteq S^{-2}\left(  S^{2}\left(  C\right)  \right)  =S^{-2}\left(
C\right)  ;
\]
but $C$ and $S^{-2}\left(  S^{2}\left(  C\right)  \right)  $ have the same
measure ($S^{2}$ is $\lambda\otimes\mu$-preserving), so $\mathbf{1}_{C}$ is a
$S^{2}$-invariant function a.s.. Since $\mathbf{1}_{C}=\pi_{2}^{\ast
}(\mathbf{1}_{\pi_{2}(C)})$ a.s. (the fibers of $C$ have all length $1$ a.s.),
$\pi_{2}\left(  C\right)  $ has full measure by Proposition
\ref{propositon characteristic function invariant}. Therefore, $\lambda
\otimes\mu\left(  C\right)  =1$ and
\[
\lim_{n\rightarrow\infty}S^{2n}\left(  B\right)  =1.
\]
This proves that $S^{2}$ is exact. In general, we have%
\begin{align*}
S^{2n}\left(  B\right)   &  \subseteq S^{-1}\left(  S\left(  S^{2n}\left(
B\right)  \right)  \right)  \subseteq S^{-1}\left(  S^{2n+1}\left(  B\right)
\right) \\
S^{2n+1}\left(  B\right)   &  \subseteq S^{-1}\left(  S\left(  S^{2n+1}\left(
B\right)  \right)  \right)  \subseteq S^{-1}\left(  S^{2(n+1)}\left(
B\right)  \right)
\end{align*}
so that, using again that $S$ is $\lambda\otimes\mu$-preserving,%
\[
\lambda\otimes\mu\left(  S^{2n}\left(  B\right)  \right)  \leq\lambda
\otimes\mu\left(  S^{2n+1}\left(  B\right)  \right)  \leq\lambda\otimes
\mu\left(  S^{2(n+1)}\left(  B\right)  \right)  ,
\]
which proves that $\left\{  \lambda\otimes\mu\left(  S^{n}\left(  B\right)
\right)  \right\}  _{n\in\mathbb{N}}$ is increasing and converges to $1$,
i.e., $S$ is also exact.\smallskip
\end{proof}

To sum up, the skew-type representation $S:[0,1)\times\lbrack0,\pi
]\rightarrow\lbrack0,1)\times\lbrack0,\pi]$ associated to the random map
(\ref{eq tau maps}) and (\ref{eq probabilities model}) is exact which implies,
by (\ref{eq 33}), that $\nu^{(n)}\rightarrow\mu$ as $n\rightarrow\infty$ for
any initial distribution $\nu$ on $[0,\pi]$ absolutely continuous with respect
to the law $\mu\left(  A\right)  =%
\frac12
\int_{A}\sin\left(  \theta\right)  d\theta$, $A\in\mathcal{B}\left(
[0,\pi]\right)  $. Since $\sin\left(  \theta\right)  >0$ on $\left(
0,\pi\right)  $, $\mu$ and the Lebesgue measure $\lambda$ are absolutely
continuous with respect to each other on $\left(  0,\pi\right)  $. In other
words, \textit{strong Knudsen's law }$\nu^{(n)}\rightarrow\mu$\textit{ holds
for any initial distribution absolutely continuous with respect to }$\lambda
$\textit{.}

\section{Knudsen's law. Simulations\label{section knudsen's law}}

In this section, we are going to show that numerical simulations are in
accordance with theoretical results. With this aim, we take an arbitrary
initial distribution on $\left[  0,\pi\right]  $ such as in the following
picture and we make it evolve according to our random system
(\ref{eq tau maps}) and (\ref{eq probabilities model}).%

\begin{center}
\includegraphics[
natheight=4.375100in,
natwidth=5.843500in,
height=2.6221in,
width=3.4904in
]%
{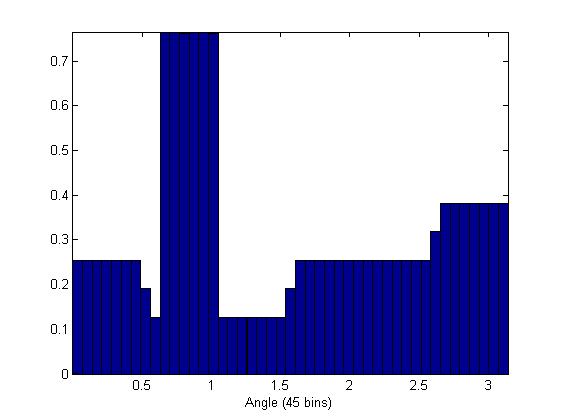}%
\\
Initial normalised distribution of particles.
\label{fig initial distribution}%
\end{center}
Since we can only simulate a finite number particles in a computer, in this
experiment we take a total amount of 30000 balls and divide $\left[
0,\pi\right]  $ in subintervals of the same length (45 in our experiment).
That is, we approximate the initial distribution by a step function. In each
subinterval, we put the proportion of balls according to the probability
density function above. The initial angle associated to any of those balls is
the middle point of the interval where they fall. The simulation then goes as
follows. At the $i$-\textit{th} step, we take the $j$-\textit{th} particle
with angle $\theta_{j}^{(i)}$ and a random number $y_{j}^{(i)}$ uniformly
distributed between $0$ and $1$, one for each particle. If $(y_{j}%
^{(i)},\theta_{j}^{(i)})\in J_{k}$, then $\theta_{j}^{(i+1)}=\tau_{k}%
(\theta_{j}^{(i)})$, where $\{\tau_{k}\}_{k=1,...,4}$ are as in
(\ref{eq tau maps}), and so on. After a long number of iterations, the
distribution stabilises as follows:
\begin{center}
\includegraphics[
trim=0.000000in -0.011480in 0.000000in 0.011480in,
natheight=4.375100in,
natwidth=5.843500in,
height=2.674in,
width=3.5596in
]%
{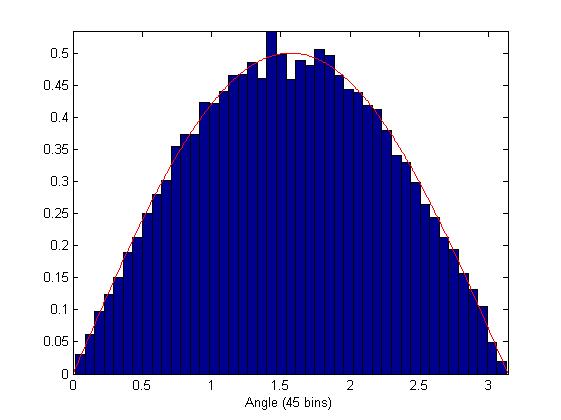}%
\\
Final normalised distribution of particles. In red, $\frac12 \sin(x)$.
\label{fig final distribution}%
\end{center}
As we can see, the outline of the final distribution tends to the graph of $%
\frac12
\sin\left(  x\right)  $. The small inaccuracy is explained by the fact that
only a finite (i.e., a discrete) number of initial angles is considered. The
smaller the subintervals in which $\left[  0,\pi\right]  $ is divided are, the
better the final distribution approximates $%
\frac12
\sin\left(  x\right)  $. We have repeated this experiment over several initial
distributions obtaining always similar results, which experimentally confirms
the validity of Knudsen's law for our model.

\appendix

\section{Appendix}

\begin{proof}
[Proof of Lemma \ref{lemma slices go to zero}]To start with, will prove that
\[
S\left(  I_{x}\bigcap J_{i_{1}\cdots i_{n}}\right)  =I_{\tau_{i_{n}}%
(x)}\bigcap J_{i_{1}\cdots i_{n-1}}.
\]
From the very definition of $S$ and $J_{i_{1}\cdots i_{n}}$, we have%
\[
S\left(  I_{x}\bigcap J_{i_{1}\cdots i_{n}}\right)  \subseteq I_{\tau_{i_{n}%
}(x)}\bigcap J_{i_{1}\cdots i_{n-1}}.
\]
On the other hand, let $\omega\in I_{\tau_{i_{n}}(x)}\bigcap J_{i_{1}\cdots
i_{n-1}}$. Then $\omega=(y,\tau_{i_{n}}(x))$ for some $y\in\lbrack0,1)$. Let
$\omega^{\prime}=(p_{i_{n}}(x)y+%
{\textstyle\sum\nolimits_{j=1}^{i_{n}-1}}
p_{j}(x),x)$. It is obvious that $\omega^{\prime}\in J_{i_{n}}$ as $p_{i_{n}%
}(x)>0$ because $I_{x}\bigcap J_{i_{1}\cdots i_{n}}$ is not empty. Moreover,
$S\left(  \omega^{\prime}\right)  =\omega$. Therefore $\omega\in S\left(
I_{x}\bigcap J_{i_{1}\cdots i_{n}}\right)  $ and%
\[
S\left(  I_{x}\bigcap J_{i_{1}\cdots i_{n}}\right)  \supseteq I_{\tau_{i_{n}%
}(x)}\bigcap J_{i_{1}\cdots i_{n-1}}.
\]
Identifying both fibers $I_{x}\bigcap J_{i_{1}\cdots i_{n}}$ and
$I_{\tau_{i_{n}}(x)}\bigcap J_{i_{1}\cdots i_{n-1}}$ as subsets of $[0,1)$,
the map%
\begin{equation}%
\begin{array}
[c]{rcl}%
S_{x}:I_{x}\bigcap J_{i_{1}\cdots i_{n}}\subset\lbrack0,1) & \longrightarrow &
I_{\tau_{i_{n}}(x)}\bigcap J_{i_{1}\cdots i_{n-1}}\\
y & \longmapsto & \varphi_{i_{n}}(y,\tau_{i_{n}}(x))=\frac{1}{p_{i_{n}}%
(x)}\left(  y-%
{\textstyle\sum\nolimits_{j=1}^{i_{n}-1}}
p_{j}(x)\right)  ,
\end{array}
\label{eq 37}%
\end{equation}
is then a diffeomorphism. After having identified $I_{x}\bigcap J_{i_{1}\cdots
i_{n}}$ as a subset of $[0,1)$, we can compute its Lebesgue measure
$\lambda\left(  I_{x}\bigcap J_{i_{1}\cdots i_{n}}\right)  $. Consequently,%
\begin{gather}
\lambda\left(  I_{x}\bigcap J_{i_{1}\cdots i_{n}}\right)  =\int_{_{I_{x}%
\bigcap J_{i_{1}\cdots i_{n}}}}d\lambda=\int_{I_{\tau_{i_{n}}(x)}\bigcap
J_{i_{1}\cdots i_{n-1}}}\frac{d(S_{x})^{-1}}{dy}d\lambda\nonumber\\
=\int_{I_{\tau_{i_{n}}(x)}\bigcap J_{i_{1}\cdots i_{n-1}}}p_{i_{n}}%
(x)d\lambda=p_{i_{n}}(x)\int_{I_{\tau_{i_{n}}(x)}\bigcap J_{i_{1}\cdots
i_{n-1}}}d\lambda\nonumber\\
=p_{i_{n}}(x)\lambda\left(  I_{\tau_{i_{n}}(x)}\bigcap J_{i_{1}\cdots i_{n-1}%
}\right)  .\label{eq 19}%
\end{gather}
Applying iteratively (\ref{eq 19}), we obtain%
\[
\lambda\left(  I_{x}\bigcap J_{i_{1}\cdots i_{n}}\right)  =p_{i_{n}%
}(x)p_{i_{n-1}}\left(  \tau_{i_{n}}(x)\right)  \cdots p_{1}(\tau_{i_{2}}%
\circ\cdots\circ\tau_{i_{n}}(x)).
\]
\medskip
\end{proof}

\noindent\textbf{Acknowledgements.} The authors would like to thank Wael
Bahsoun and Renato Feres for their enlightening comments and suggestions.

\end{document}